\documentclass{amsart}

\usepackage{amsmath,amssymb, amsthm}
\usepackage[french]{babel}

\usepackage{hyperref}
\usepackage{url}
\usepackage{enumerate}
\usepackage{graphicx}
\usepackage{caption, subcaption}

\theoremstyle{plain}
\newtheorem{theo}{Th\'eor\`eme}[section]
\newtheorem{lem}[theo]{Lemme}
\newtheorem{prop}[theo]{Proposition}
\newtheorem{conj}{Conjecture}
\newtheorem{question}[conj]{Question}
\newtheorem*{cor}{Corollaire}

\theoremstyle{definition}

\newcommand{\sub}{\mathrm{Sub}}
\newcommand{\supp}{\mathrm{supp}}

\newcommand{\eps}{\varepsilon}

\newcommand{\ovl}[1]{\overline #1}

\newcommand{\pl}{\partial}
\newcommand{\bs}{\backslash}

\newcommand{\vol}{\operatorname{vol}}
\newcommand{\Id}{\operatorname{Id}}

\newcommand{\tr}{\operatorname{tr}}
\newcommand{\otr}{\operatorname{Tr}}

\renewcommand{\hom}{\operatorname{Hom}}
\newcommand{\inj}{\operatorname{inj}}

\newcommand{\G}{\mathbf{G}}
\newcommand{\PSL}{\mathrm{PSL}}
\newcommand{\SL}{\mathrm{SL}}
\newcommand{\SU}{\mathrm{SU}}
\newcommand{\SO}{\mathrm{SO}}

\newcommand{\GL}{\mathrm{GL}}

\newcommand{\so}{\mathcal{O}}

\newcommand{\prob}{\mathrm{Prob}}

\newcommand{\isom}{\mathrm{Isom}}

\newcommand{\NN}{\mathbb N}

\newcommand{\CC}{\mathbb C}
\newcommand{\RR}{\mathbb R}
\newcommand{\ZZ}{\mathbb Z}
\newcommand{\HH}{\mathbb H}

\newcommand{\QQ}{\mathbb Q}

\author[Jean Raimbault]{Jean Raimbault \\ avec un appendice par l'auteur et Ian Biringer}

  \address{Institut de Math\'ematiques de Toulouse ; UMR 5219, Universit\'e de Toulouse, CNRS \\ UPS IMT, F-31062 Toulouse Cedex 9, France}
  \email{Jean.Raimbault@math.univ-toulouse.fr}

\title{G\'eom\'etrie et topologie des vari\'et\'es hyperboliques de grand volume}

\begin{document}

\maketitle

\begin{abstract}
Cet article est un survol autour de deux pr\'epublications r\'ecentes \cite{7samurai} et \cite{moi1}, qui se posent la question de l'\'etude de certains invariants topologiques et g\'eom\'etriques dans des suites d'espaces localement sym\'etriques dont le volume tend vers l'infini. On donne aussi quelques applications \`a divers mod\`eles de surfaces al\'eatoires. 
\end{abstract}

\setcounter{tocdepth}{1}
\tableofcontents
\setcounter{tocdepth}{2}


\section{Introduction}
\label{intro}
Dans cet article en grande partie expositoire on s'int\'eressera \`a quelques aspects de la question suivante : on prescrit des bornes g\'eom\'etriques locales (par exemples, dans le cadre Riemannien, des bornes sur la courbure sectionnelle ou de Ricci, ou plus drastiquement on prescrit la classe d'isom\'etrie locale), et on se demande si ces conditions imposent des contraintes {\it globales} sur la g\'eom\'etrie, la topologie et leurs relations. Un exemple classique et frappant de r\'esultat allant dans cette direction est le th\'eor\`eme de Marcel Berger et Daniel Grove et Katsuhiro Shiohama selon lequel, si une vari\'et\'e riemannienne compacte a toutes ses courbures sectionelles sup\'erieures \`a 1 et un diam\`etre strictement plus grand que $\pi/2$ alors elle est en fait diff\'eomorphe \`a la sph\`ere de m\^eme dimension (cf. \cite[Theorem 81]{Petersen}). Un autre exemple est le r\'esultat de Gromov qui donne des bornes absolues (ne d\'ependant que de la dimension) sur le rang du groupe fondamental et les nombres de Betti (\`a coefficients arbitraires) d'une vari\'et\'e \`a courbure sectionelle positive (cf. \cite[Theorem 86]{Petersen}). 

Nous nous int\'eresserons ici plut\^ot \`a la classe des espaces \`a courbure n\'egative, pour lesquels la situation est tr\`es diff\'erente. Par exemple, m\^eme si on impose une courbure sectionnelle constante il existe en toute dimension donn\'ee une infinit\'e de types topologiques, en fait les nombres de Betti peuvent \^etre arbitrairement grands (cf. \cite{Lubotzky_free_quotients}). En revanche, sous la m\^eme condition de courbure n\'egative constante et sauf en dimension trois\footnote{Le cas de la dimension trois est assez diff\'erent mais bien compris gr\^ace aux travaux de Thurston sur la chirurgie de Dehn hyperbolique.}, pour un $V>0$ donn\'e il n'existe \`a diff\'eomorphisme pr\`es qu'un nombre fini vari\'et\'es $M$ satisfaisant en plus \`a une borne $\vol(M)\le V$. On s'int\'eresse alors naturellement aux contraintes impos\'ees par le volume sur la g\'eom\'etrie globale et la topologie. Par example un th\'eor\`eme aussi d\^u \`a Gromov affirme que les nombres de Betti sont born\'es lin\'eairement en le volume (pour une vari\'et\'e ayant des courbures sectionnelles entre $-1$ et 0 et pas de facteur euclidien, avec une constante d\'ependant uniquement de la dimension---cf. \cite[Theorem 2]{Ballmann_Gromov_Schroeder}).

Les r\'esultats sur lesquels nous nous attarderons dans ce survol concerneront les espaces localement sym\'etriques \`a courbure n\'egative sans facteurs euclidiens, autrement dit les quotients $\Gamma\bs X$ o\`u $X=G/K$ est l'espace sym\'etrique associ\'e \`a un groupe de Lie semisimple $G$ et un sous-groupe compact maximal $K\subset G$, et $\Gamma$ un r\'eseau de $G$, c'est-\`a-dire un sous-groupe discret tel que $G/\Gamma$ ait une mesure bor\'elienne finie et $G$-invariante. Dans ce cadre on s'int\'eresse \`a la g\'eom\'etrie globale principalement \`a travers les invariants m\'etriques $\vol(M_{\le R})/\vol(M)$ o\`u $M_{\le R}$ d\'esigne la partie $R$-mince\footnote{Certains utilisent plut\^ot ``fin'' pour traduire l'anglais {\it thin}, mais le terme de ``partie fine'' me semble maladroit en fran\c cais.}, autrement dit le sous-ensemble des points de $M$ autour desquels le rayon maximal d'une boule plong\'ee est plus petit que $R/2$ (de mani\`ere \'equivalente, il existe un lacet homotopiquement non-trivial sur $M$ passant par ce point et de longueur inf\'erieure \`a $R$). On s'int\'eressera \`a l'\'etude de suites de vari\'et\'es o\`u cette quantit\'e tend vers 0, \`a la fois pour ses cons\'equences topologiques non-triviales (cf. \ref{multlim}, \ref{topo_dim3}) et pour l'abondance de situations o\`u cette condition est r\'ealis\'ee (cf. \ref{applications}, \ref{random_surfaces} et \ref{higher_dim}). 

Cette condition peut \^etre interpr\'et\'ee comme la convergence vers le rev\^etement universel $X$ dans une certaine compactification de l'espace des r\'eseaux $\Gamma$ de $G$. Pour d\'efinir cette compactification il faut introduire la topologie de Benjamini--Schramm sur les espaces m\'etriques al\'eatoires point\'es, qui est une version probabiliste de la topologie bien connue de Gromov--Hausdorff point\'ee. Les objets que l'on doit alors consid\'erer ne sont alors plus des sous-groupes discrets de $G$, mais des mesures de probabilit\'es invariantes par conjugaison sur l'espace des sous-groupes ferm\'es de ce dernier. Ces sous-groupes al\'eatoires invariants ont \'et\'e introduits par Mikl\'os Ab\'ert, Bal\'int Vir\'ag et Yair Glasner dans \cite{AGV}, et ont r\'ecemment fait l'objet d'assez nombreux travaux (cf. \cite{Eisenmann_Glasner} pour un exemple et de nombreuses autres r\'ef\'erences r\'ecentes). Dans cet article on essaiera plut\^ot de d\'egager les cons\'equences g\'eom\'etriques de l'\'etude des sous-groupes al\'eatoires des groupes de Lie, qui constituent un puissant outil pour l'\'etude des vari\'et\'es de grand volume, t\'emoins les th\'eor\`emes \ref{BS_rangsup} et \ref{kleinian_arit_conv} ci-dessous. Le pendant plus g\'eom\'etrique de ces derniers, qui sera tr\`es peu explor\'e ici, est donn\'e par les mesures de probabilit\'e unimodulaires sur les espaces de vari\'et\'es riemanniennes ; on ref\`ere \`a \cite{Abert_Biringer} pour plus de d\'etails. Citons aussi l'article \cite{NPS}, qui \'etudie dans un langage plus g\'eom\'etrique la topologie des limites de certaines suites de m\'etriques riemanniennes sur une vari\'et\'e donn\'ee. 

On peut encore restreindre le cadre d'\'etude en s'int\'eressant aux vari\'et\'es arithm\'etiques, qui sont une famille particuli\`ere construites \`a partir de r\'eseaux g\'en\'eralisant l'exemple $\PSL_2(\ZZ)\subset\PSL_2(\RR)$. Une sous-famille distingu\'ee est donn\'ee par les r\'eseaux dits de congruence\footnote{Que l'on ne d\'efinira pas ici ; mais ils contiennent en particulier les r\'eseaux maximaux, et voir aussi \ref{rev_cong} pour plus d'exemples.} Ces derniers ont souvent des propri\'et\'es g\'eom\'etriques particuli\`erement r\'eguli\`eres : par exemple leurs constantes de Cheeger sont uniform\'ement born\'ees inf\'erieurement (un r\'esultat remontant \`a Atle Selberg pour les surfaces hyperboliques, prouv\'e dans la plus grande g\'en\'eralit\'e par Laurent Clozel dans \cite{Clozel_tau}). Leur origine arithm\'etique permet de plus d'attaquer les probl\`eme g\'eom\'etriques en utilisant les outils de la th\'eorie des nombres : un exemple frappant est le probl\`eme du volume minimal pour les vari\'et\'es hyperboliques, qui est encore compl\`etement ouvert en grandes dimensions mais qui dans le cadre des vari\'et\'es arithm\'etiques est sinon r\'esolu compl\`etement, au moins plus pr\`es d'une telle r\'esolution (on ref\`ere \`a \cite{Belolipetsky_icm} pour un survol r\'ecent du sujet). Pour ce qui est de la convergence de Benjamini--Schramm vers le rev\^etement universel on peut l'\'etablir directement (sans utiliser les sous-groupes al\'eatoires) dans plusieurs cas : voir les th\'eor\`emes \ref{BSconv_congruence} et \ref{kleinian_arit_conv} ci-dessous. 

Les r\'esultats sur la topologie de Benjamini--Schramm \'evoqu\'es ci-dessus ne sont pas nouveaux, et proviennent pour la plupart des articles \cite{7samurai} et \cite{moi1}. Nous apportons aussi dans cet article quelques observations et r\'esultats originaux :
\begin{itemize}
\item Le th\'eor\`eme \ref{IRS_dim2_theo} donne une classification topologique des sous-groupes al\'eatoires invariants de $\PSL_2(\RR)$ (d\'emontr\'ee dans l'appendice \ref{IRS_dim2} \'ecrit avec I. Biringer). 
\item En \ref{random_surfaces} on donne une interpr\'etation de r\'esultats de Maryam Mirzakhani \cite{Mirzakhani} et de Robert Brooks et Eran Makover \cite{Brooks_Makover} comme r\'esultats de convergence, et leur cons\'equences pour le spectre des surfaces al\'eatoires.  
\end{itemize}

\medskip

Cet article est organis\'e comme suit : en \ref{BSconv} on donne un r\'esum\'e de la plupart des r\'esultats et notions contenus dans \cite{7samurai} : on introduit en \ref{Benjamini_Schramm_def}, dans un cadre g\'en\'eral, la notion de convergence de Benjamini--Schramm, que l'on sp\'ecialise ensuite au cadre des espaces localement sym\'etriques. On indique les liens pr\'ecis avec les sous-groupes al\'eatoires dans la section \ref{BS_sym}. La section suivante \ref{IRS} donne quelques propri\'et\'es de ces derniers, puis en \ref{multlim} on montre comment d\'eduire de la convergence de Benjamini--Schramm des r\'esultats de multiplicit\'es limites pour les valeurs propres des laplaciens. La derni\`ere section \ref{applications} de cette partie pr\'esente les deux cas de convergence \'etudi\'es dans \cite{7samurai}. La seconde partie est centr\'ee sur les vari\'et\'es hyperboliques r\'eelles : on commence en \ref{IRS_hyperboliques} par un expos\'e des r\'esultats connus sur les sous-groupes invariants des groupes $\SO(n,1)$ (y compris le th\'eor\`eme d\'emontr\'e dans l'appendice \ref{IRS_dim2}). Puis la section \ref{random_surfaces} interpr\`ete quelques r\'esultats connus sur certains mod\`eles al\'eatoires de surfaces hyperboliques \`a la lumi\`ere de la convergence de Benjamini--Schramm. Enfin la section \ref{higher_dim} essaie de d\'egager quelques id\'ees sur la convergence des vari\'et\'es hyperboliques en plus grandes dimensions (on s'y permet d'\^etre plus sp\'eculatif que dans le reste de l'article). 

\subsection{Remerciements} Je suis redevable \`a Bram Petri de nombreuses remarques et corrections sur la section concernant les surfaces al\'eatoires. Je voudrais aussi remercier Ian Biringer pour m'avoir expliqu\'e son travail avec Mikl\'os Ab\'ert.


\section{Convergence de Benjamini--Schramm}
\label{BSconv}
\subsection{G\'en\'eralit\'es}
\label{Benjamini_Schramm_def}

Soit ${\mathcal X}$ l'ensemble des espaces m\'etriques localement compacts point\'es (ou plut\^ot de leurs classes d'isom\'etrie) ; on va d\'efinir dans cette section des topologies sur ${\mathcal X}$ et $\prob\left({\mathcal X}\right)$ respectivement. La topologie sur ${\mathcal X}$ est bien connue et appel\'ee topologie de Gromov-Hausdorff ; des comptes-rendus d\'etaill\'es en sont donn\'es par exemple dans \cite[Chapter 10]{Petersen} ou \cite[Chapter 3]{Gromov_metric}. Sur le sous-ensemble de ${\mathcal X}$ form\'e des espaces compacts point\'es on d\'efinit une distance comme suit : si $A,B$ sont des sous-espaces compacts d'un espace m\'etrique $Z$ leur distance de Hausdorff $d_H(A,B)$ est l'infimum des $\eps>0$ tels que $A,B$ soient chacun contenu dans le $\eps$-voisinage (dans $Z$) de l'autre ; si $(X,x),(Y,y)$ sont des espaces compacts point\'es on pose alors
$$
d((X,x),(Y,y)) = \inf_{Z,\phi}\left(d_H(\phi_1(X),\phi_2(Y)) + d_Z(\phi_1(x),\phi_2(y))\right)
$$
o\`u l'infimum est pris sur l'ensemble des espaces m\'etriques compacts $Z$ et des paires de plongements isom\'etriques $\phi_1:X\to Z,\phi_2:Y\to Z$ ; on v\'erifie que cela d\'efinit bien une distance. En particulier, pour tout $R>0$ on obtient une topologie sur l'espace ${\mathcal X}_R$ des boules de rayon $R$ point\'ees en leur centre (i.e. les espaces point\'es $(X,x)$ tels que $d(x,y)\le R$ pour tout $y\in X$). On a une application ${\mathcal X}\to {\mathcal X}_R$ d\'efinie par $(X,x)\mapsto B_X(x,R)$, et on d\'efinit la topologie de Gromov--Hausdorff sur ${\mathcal X}$ comme la plus faible qui rende continues toutes ces applications. Il n'est alors pas dur de voir qu'une suite $(X_n,x_n)$ converge vers $(X,x)$ si et seulement si pour touts $R,\eps>0$, et pour $n$ assez grand, la boule $B_{X_n}(x_n,R)$ est $(1+\eps,\eps)$-quasi-isom\'etrique \`a $B_X(x,R)$. 

L'espace ${\mathcal X}$ muni de cette topologie n'est pas compact, cependant il contient de nombreux sous-ensembles qui le sont : typiquement, si on precrit des bornes locales sur la g\'eom\'etrie d'une suite d'espaces on obtient des parties relativement compactes dans ${\mathcal X}$. Pour des exemples riemanniens on ref\`ere par exemple \`a \cite[Chapitre 10, 3.4]{Petersen}. Un exemple plus simple est donn\'e par l'ensemble des graphes dans lesquels la valence de chaque sommet est uniform\'ement born\'ee.

La topologie introduite ci-dessus a donc de bonnes propri\'et\'es de compacit\'e, mais en g\'en\'eral elle est compl\`etement aveugle aux propri\'et\'es globales des espaces \'etudi\'es. Pour rem\'edier \`a cela, tout en conservant ses propri\'et\'es d\'esirables, on va passer \`a l'\'etude de $\prob\left({\mathcal X}\right)$, l'espace des mesures de probabilit\'es bor\'eliennes sur ${\mathcal X}$ que l'on munit de la topologie de la convergence faible des mesures. Cette derni\`ere a \'et\'e nomm\'ee dans \cite{7samurai} topologie de Benjamini--Schramm, \`a la suite du travail pionnier de ces auteurs sur les graphes finis dans \cite{Benjamini_Schramm}. Dans le reste de cet article on s'int\'eressera surtout \`a l'\'etude de cette notion dans le cadre des espaces localement sym\'etriques. Dans le cadre plus large des vari\'et\'es riemanniennes, il y a essentiellement deux fa\c cons naturelles de construire des \'el\'ements de $\prob({\mathcal X})$: 
\begin{itemize}
\item[(i)] Si $M$ est une vari\'et\'e de volume fini, on peut consid\'erer la mesure de probabilit\'e sur $\mathcal X$ obtenue en pointant $M$ en un point choisi al\'eatoirement pour la mesure de probabilit\'e induite par $\nu = d\vol/\vol(M)$ sur $M$ (i.e. le pouss\'e en avant de $\nu$ par l'application $M\to{\mathcal X},\, x\mapsto(M,x)$). 
\item[(ii)] Si $X$ est un espace topologique munie d'un feuilletage $\mathcal F$ par vari\'et\'es riemanniennes et d'une mesure de probabilit\'e $\nu$ (de pr\'ef\'erence invariante par rapport au feuilletage), on a une application $X\to{\mathcal X}$ donn\'ee par $x\mapsto ({\mathcal F}_x,x)$ et on peut consid\'erer le pouss\'e en avant de $\nu$. 
\end{itemize}
Evidemment, (i) est un cas particulier de (ii) (o\`u le feuilletage est trivial). On remarque que la construction (ii) est r\'eminescente des feuilletages apparaissant comme limites de rev\^etements finis dans \cite{Bergeron_Gaboriau}. Des examples plus sp\'ecifiques de ces deux constructions (dans le cadre localement sym\'etrique) seront donn\'ees plus bas. Une \'etude plus syst\'ematique des vari\'et\'es riemanniennes point\'ees al\'eatoires est entreprise dans \cite{Abert_Biringer}.


\subsection{Espaces localement sym\'etriques et sous-groupes al\'eatoires invariants}
\label{BS_sym}

\subsubsection{Topologie de Chabauty}

Soit $G$ un groupe de Lie semisimple ; on note $\sub_G$ l'espace des sous-groupe ferm\'es de $G$ muni de la topologie de Chabauty\footnote{Restriction de la topologie de Hausdorff ; cette topologie a \'et\'e initialement consid\'er\'ee par Claude Chabauty dans \cite{Chabauty} ; une description des ouverts est donn\'ee dans \cite[Section 2]{7samurai}.}. On a alors les propri\'et\'es \'el\'ementaires suivantes : 
\begin{itemize}
\item[(i)] L'espace $\sub_G$ est compact ;
\item[(ii)] Le point $G$ est isol\'e dans $\sub_G$ ; 
\end{itemize}

Soit $K$ un sous-groupe compact maximal, $X=G/K$ muni de la m\'etrique riemannienne $G$-invariante qui en fasse un espace sym\'etrique ; on note $x_0$ l'unique point fixe de $K$ dans $x$, et si $x\in X$ et $\Gamma$ est un sous-groupe discret de $G$ on note $\bar x$ l'image de $x$ dans $\Gamma\bs X$. Le r\'esultat suivant est bien connu (cf. \cite[Proposition 3.3]{7samurai}). 

\begin{lem}
Soient $\Gamma_n, n\ge 1$ et $\Lambda$ des sous-groupes discrets de $G$. La suite  $(\Gamma_n\bs X,\bar x_0)$ converge vers $(\Lambda\bs X,\bar x_0)$ dans $\mathcal X$ si et seulement si $\Gamma_n$ converge vers $\Lambda$ dans la topologie de Chabauty. 
\label{Chabauty_GH}
\end{lem}


\subsubsection{Sous-groupes al\'eatoires invariants}

Un {\it sous-groupe al\'eatoire} de $G$ est par d\'efinition une mesure de probabilit\'e bor\'elienne sur $\sub_G$ qui est invariante par les applications $H\mapsto gHg^{-1}$ pour $g\in G$. Dans la suite on abr\'eviera souvent ce nom en IRS, pour ne pas surcharger le texte. Un IRS $\mu$ de $G$ est dit ergodique si l'action de $G$ sur $(\sub_G,\mu)$ l'est. De mani\`ere (non-trivialement) \'equivalente, un IRS ergodique est le tir\'e en arri\`ere d'une mesure de probabilit\'e invariante dans une action ergodique de $G$ (cf. \cite[Theorem 2.4]{7samurai}). Les exemples fondamentaux d'IRS ergodiques dans les groupes de Lie sont les suivants:
\begin{itemize}
\item Les masses de Dirac sur les sous-groupes normaux ; en particulier, on a toujours au moins deux IRS ergodiques de $G$, not\'es $\delta_G$ et $\delta_{\{\Id\}}$ et qui sont les masses de Dirac support\'ees respectivement sur $G$ lui-m\^eme et sur le sous-groupe trivial. 
\item Si $\Gamma$ est un r\'eseau (sous-groupe discret de covolume fini pour la mesure de Haar), soit $\mu$ la mesure de probabilit\'e $G$-invariante sur $G/\Gamma$, $\Phi$ l'application 
$$
G/\Gamma\to\sub_G, \, g\Gamma\mapsto g\Gamma g^{-1}. 
$$
Alors le pouss\'e en avant $\Phi_*\mu$ est un IRS ergodique (support\'e sur les conjug\'es de $\Gamma$ dans $G$ ; dans la suite on le notera $\mu_\Gamma$. 
\end{itemize}

Une autre construction de sous-groupes al\'eatoires invariants importante pour la suite est l'induction depuis un r\'eseau $\Gamma\le G$ : si $\nu$ est un IRS de $\Gamma$, on peut construire un IRS $\mu_\nu$ de $G$ support\'e sur les conjugu\'es par $G$ des groupes dans le support de $\nu$. La construction est d\'etaill\'ee dans \cite[11.1]{7samurai}, informellement elle revient \`a choisir d'abord un sous-groupe $\mu_\Gamma$-al\'eatoire de $G$ puis un sous-groupe $g_*\nu$-al\'eatoire de ce conjugu\'e par $g$ de $\Gamma$. 
 \label{IRS_induit}En particulier, si $\Lambda\le\Gamma$ est un sous-groupe distingu\'e alors on a une application $G/\Gamma\to\sub_G, \, g\Gamma\mapsto g\Lambda g^{-1}$ et le pouss\'e en avant de la mesure de la mesure de probabilit\'e invariante sur $G/\Gamma$ est un IRS de $G$ que l'on notera $\mu_\Lambda$, support\'e sur les conjugu\'es de $\Lambda$ de $G$. On donnera des exemples d'IRS des groupes $\SO(d,1)$ utilisant les sp\'ecificit\'es de la g\'eom\'etrie hyperbolique (en particulier en dimensions $d=2,3$) dans la deuxi\`eme partie de cet article. 


\subsubsection{IRS et convergence de Benjamini--Schramm}
\label{IRS_BSconv}

Si $\nu$ est un IRS de $G$ support\'e sur des sous-groupes discrets, on obtient une mesure de probabilit\'e $m(\nu)$ sur $\mathcal X$ en poussant en avant $\nu$ par l'application $\sub_G\to {\mathcal X}, \, \Gamma\mapsto (\Gamma\bs X, \bar x_0)$ (l'application est bien d\'efinie sur le support de $\nu$). On a alors le r\'esultat suivant, version probabiliste du lemme \ref{Chabauty_GH} (cf. \cite[Corollary 3.4]{7samurai}). 

\begin{lem}
Si $\nu_n,n\ge 1$ et $\nu_\infty$ sont des IRS de $G$, on a $m(\nu_n)\to m(\nu_\infty)$ dans $\prob({\mathcal X})$ si et seulement si $\nu_n\to\nu_\infty$ dans l'espace des IRS de $G$. 
\label{IRS_BS}
\end{lem}

Si $M$ est une $X$-vari\'et\'e de volume riemannien fini alors l'IRS de $G$ d\'efini par $\mu_\Gamma$, o\`u $\Gamma$ est l'image de $\pi_1(M)$ dans $G$ par une application de monodromie, ne d\'epend pas du choix de point base ; on le notera $\mu_M$. Par le lemme ci-dessus, on peut caract\'eriser la convergence vers $X$ dans le sens introduit en \ref{intro} comme suit. 

\begin{lem}
Soient $X=G/K$ un espace localement sym\'etrique (de type non-compact, sans facteur euclidien) et $M_n$ une suite de vari\'et\'es localement isom\'etriques \`a $X$, de volume fini. Alors on a 
\begin{equation}
\forall R>0, \, \frac{ \vol\left( x\in M_n:\: \inj_x(M_n) \le R \right) }{ \vol M_n } \xrightarrow[n\to\infty]{} 0
\label{BS_trivial}
\end{equation}
si et seulement si les IRS $\mu_{M_n}$ convergent vers $\delta_{\{\Id\}}$ (au sens de la convergence faible des mesures). 
\end{lem}


\subsection{Propri\'et\'es des sous-groupes al\'eatoires invariants}
\label{IRS}

\subsubsection{Zariski-densit\'e}

La propri\'et\'e des sous-groupes al\'eatoires invariants d'un groupe semisimple la plus importante pour nous est le th\'eor\`eme suivant \cite[Theorem 2.6]{7samurai}, une g\'en\'eralisation du classique th\'eor\`eme de densit\'e de Borel. 

\begin{theo}
Soit $G$ un groupe de Lie simple ; alors les IRS non-atomiques de $G$ sont support\'es sur les sous-groupes discrets et Zariski-denses de $G$. 
\label{Borel_density}
\end{theo}

En rang un on a une propri\'et\'e plus forte (le r\'esultat suivant est un cas particulier de \cite[Proposition 11.3]{7samurai}).  

\begin{theo}
Soit $X$ un espace sym\'etrique irr\'eductible de rang un, $G=\isom(X)^\circ$ et $\mu$ un IRS de $G$ sans atomes. Alors $\mu$-presque tout sous-groupe a un ensemble limite \'egal \`a $\pl X$. 
\label{limit_set}
\end{theo}

Toujours en se restreingant au rang un, on peut d\'eduire de la Zariski-densit\'e le crit\`ere de BS-convergence suivant \cite[Proposition 2.3]{moi1}. 

\begin{prop}
Soit $X$ un espace localement sym\'etrique irr\'eductible de rang un et $M_n$ une suite de vari\'et\'es de volume fini localement isom\'etiques \`a $X$. Alors $M_n$ est convergente au sens de Benjamini--Schramm vers $X$ si et seulement si on a, pour tout $R>0$ : 
$$
\frac{|\{\text{g\'eod\'esiques de longueur } \le R \text{ sur } M_n\}|} {\vol M_n} \xrightarrow[n\to+\infty]{} 0. 
$$
\label{geod_BSconv}
\end{prop}


\subsubsection{Th\'eor\`eme de Nevo--St\"uck--Zimmer}

Pour les groupes simples de rang sup\'erieur on a une description compl\`ete des IRS, d\^ue \`a Garrett St\"uck et Robert Zimmer \cite{Stuck_Zimmer} (une erreur dans la preuve d'un r\'esultat interm\'ediaire crucial a \'et\'e corrig\'ee par Amos Nevo et Zimmer \cite{Nevo_Zimmer}). 

\begin{theo}[Nevo--St\"uck--Zimmer]
Soit $G$ un groupe de Lie semisimple, de rang r\'eel sup\'erieur \`a 2, dont tous les facteurs ont la propri\'et\'e (T) de Kazhdan. Soit $\mu$ un sous-groupe al\'eatoire invariant de $G$ qui soit ergodique, sans atome et irr\'eductible. Alors il existe un r\'eseau $\Gamma$ de $G$ tel que $\mu=\mu_\Gamma$ (d\'efini plus haut).
\label{Nevo_Stuck_Zimmer}
\end{theo}

Si tous les facteurs sont de rang sup\'erieur, les IRS (pas forc\'ement irr\'eductibles) sont les produits des IRS des facteurs simples (auxquels le th\'eor\`eme ci-dessus s'applique). Le r\'esultat de St\"uck et Zimmer traite aussi le cas du produit d'au moins deux groupes de rang un qui poss\`edent la propri\'et\'e (T) de Kazhdan. Le cas g\'en\'eral du produit de deux groupes de rang un fait l'objet de travaux en cours de Arie Levit (qui a aussi d\'emontr\'e une version nonarchim\'edienne de St\"uck--Zimmer). 


\subsection{Convergence de Benjamini--Schramm et multiplicit\'es limites}
\label{multlim}

\subsubsection{Uniforme discr\'etion}

On dit qu'un sous-ensemble $S\subset\sub_G$ est uniform\'ement discret s'il existe un voisinage ouvert $U$ de $\Id$ dans $G$ tel que $\Lambda\cap U=\{\Id\}$ pour tout $\Lambda\in S$. Un ensemble $\mathcal M$ d'IRS de $G$ est dit uniform\'ement discret si $\bigcup_{\mu\in\mathcal M} \supp(\mu)$ l'est ; en particulier, si $\Gamma_n$ est une suite de r\'eseaux de $G$ elle est uniform\'ement discr\`ete si et seulement si les $\Gamma_n$ sont cocompacts et leur systole est minor\'ee par une constante positive. 


\subsubsection{Noyaux de la chaleur, valeurs propres du Laplacien et nombres de Betti}

Si $M$ est une vari\'et\'e riemannienne compacte, les op\'erateurs de Hodge--Laplace $\Delta^p[M]$ sont des op\'erateurs diff\'erentiels elliptiques d\'efinis sur les espaces $\Omega^p(M)$ de formes diff\'erentielles lisses sur $M$. Ils ne sont pas born\'es pour la structure pr\'e-hilbertienne donn\'ee par le produit scalaire $L^2$ des \'el\'ements de $\Omega^p(M)$, mais admettent une unique extension maximale comme op\'erateurs sym\'etriques positifs essentiellement autoadjoints sur l'espace de Hilbert $L^2\Omega^p(M)$ des formes de carr\'e int\'egrable. De plus, leur spectre est discret \`a multiplici\'es finies, i.e. on a des suites $0=\lambda_0<\lambda_1<\ldots$ et $m(\lambda_j,M)>0, j\ge 1$ et telles que l'espace propre $\ker(\Delta^p[M] - \lambda_j)$ soit de dimension $m(\lambda_j,M)$ et l'espace $L^2\Omega^p(M)$ soit la somme hilbertienne de ces sous-espaces. On d\'efinit la mesure spectrale normalis\'ee $\nu_M^p$ comme suit:
$$
\nu_M^p(S) = \frac 1{\vol M} \sum_{\lambda\in S} m(\lambda,M). 
$$

Le noyau de la chaleur de $M$ est un tenseur $e^{-t\Delta^p[M]}$ sur $M\times M$, tel que 
$$
e^{-t\Delta^p[M]}(x,y)\in \hom\left(\wedge^p T_xM,\wedge^p T_yM \right), 
$$
qui peut \^etre d\'efini comme la solution fondamentale \`a une \'equation de la chaleur appropri\'ee sur sur $M$ (cf. par exemple \cite[Chapter V]{Taylor_EDP}). La convolution avec $e^{-t\Delta^p[M]}$ d\'efinit un op\'erateur born\'e sur $L^2\Omega^p(M)$ (qui est aussi donn\'e par le calcul spectral appliqu\'e \`a $\Delta^p[M]$ et \`a la fonction $\lambda\mapsto e^{-t\lambda}$, d'o\`u la notation). C'est un op\'erateur \`a trace, et la formule des traces (qui est \`a peu pr\`es une cons\'equence imm\'ediate des d\'efinitions dans le cas d'une vari\'et\'e compacte) donne l'\'egalit\'e
\begin{equation}
\otr e^{-t\Delta^p[M]} := \sum_{j\ge 0} m(\lambda_j)e^{-t\lambda_j} = \int_M \tr e^{-t\Delta^p[M]}(x,x). 
\label{trace_formula_compact}
\end{equation}
Le r\'esultat suivant est prouv\'e dans \cite{Donnelly_towers}

\begin{lem}
Si les $M_n,\, n\ge 1$ sont des vari\'et\'es riemanniennes, et s'il existe une mesure bor\'elienne $\nu$ sur $[0,+\infty[$ telle que l'on ait la limite 
$$
\frac{\otr e^{-\Delta^p[M_n]}}{\vol M_n} \xrightarrow[n\to\infty]{} \int_0^{+\infty} e^{-t\lambda} d\nu(\lambda)
$$
pour tout $t>0$, alors la suite des mesures spectrales $\nu_{M_n}^p$ converge faiblement vers $\nu$. 
\label{trace_multiplicities}
\end{lem}

Si $X=G/K$ est un espace sym\'etrique, le spectre n'est plus discret mais on peut quand m\^eme d\'efinir des mesures spectrales $\nu_X^p$. Le terme de droite du lemme ci-dessus est alors donn\'e par la trace ponctuelle d'un noyau invariant sur $X$ (aussi obtenu comme une solution fondamentale \`a une \'equation de la chaleur sur $X$), i.e. 
$$
\int_0^{+\infty} e^{-t\lambda} d\nu_X^p(\lambda) = \tr e^{-t\Delta^p[X]}(x,x) =: \otr^{(2)} e^{-t\Delta^p[X]}
$$
pour n'importe quel $x\in X$. On a alors le r\'esultat suivant, qui sous la condition d'uniforme discr\'etion suit facilement de la formule des traces \eqref{trace_formula_compact}, d'estim\'ees bien connues sur la taille des orbites d'un r\'eseau et sur la d\'ecroissance des noyaux de la chaleur (cf. \cite[Corollary 8.27]{7samurai}), et du lemme \ref{trace_multiplicities}. La suppression de cette hypoth\`ese pour le rang un n\'ec\'essite plus de travail (cf. \cite[Section 9]{7samurai}). 

\begin{prop}
Si $X$ est un espace sym\'etrique, $M_n$ une suite de $X$-vari\'et\'es compactes qui soit convergente au sens de Benjamini--Schramm vers $X$. Si les $M_n$ ont une systole uniform\'ement minor\'ee, ou si $X$ est de rang un, alors on a 
$$
\otr e^{-t\Delta^p[M_n]} \xrightarrow[n\to\infty]{} \otr^{(2)} e^{-t\Delta^p[X]}
$$
pour tout $t>0$. En particulier, pour touts $b> a \ge 0$ on a 
$$
\frac 1{\vol M_n} \sum_{\lambda_j\in[a,b]} m(\lambda_j,M_n) \xrightarrow[n\to\infty]{} \nu_X^p([a,b]). 
$$
\label{multlim_eigenvalues}
\end{prop}

Un autre corollaire est le r\'esultat suivant. 

\begin{cor}
Soient $X,M_n$ comme ci-dessus. Pour tout degr\'e $p\not=\dim X/2$ on a la limite :
$$
\lim_{n\to\infty} \frac{b_p(M_n)}{\vol M_n} = 0. 
$$
\end{cor}

Pour $p = \dim X/2$ la limite est calcul\'ee par le th\'eor\`eme de Chern--Gauss--Bonnet : elle ne d\'epend que de $X$, peut \^etre exprim\'ee comme $\chi(M)/\vol(M)$ pour n'importe quelle $X$-vari\'et\'e compacte $M$. Elle n'est non nulle que dans le cas o\`u $G$ contient un sous-groupe de Cartan compact (i.e. $G$ et $K$ ont m\^eme rang complexe), par exemple dans le cas o\`u $X=\HH^{2m}$ on a 
$$
\lim_{n\to+\infty} \frac{b_m(M_n)}{\vol M_n} = \frac 2 {V_{2m}}
$$
pour toute suite $M_n$ qui soit BS-convergente vers $\HH^{2m}$, o\`u $V_{2m}$ est le volume de la sph\`ere unit\'e dans $\RR^{2m+1}$. 


\subsubsection{Valeurs propres exceptionnelles}

Nous nous restreindrons dans cette section aux vari\'et\'es hyperboliques r\'eelles (bien que les r\'esultats soient valides pour des espaces localement sym\'etriques plus g\'en\'eraux). Les mesures $\nu_X^p$ pour $X=\HH^d$ sont assez bien connues ; en particulier, pour $p<d/2$ elles sont support\'ees sur l'intervalle 
$$
\left[ \left( p - \frac{d-1}2 \right)^2, +\infty \right[. 
$$
On note $\lambda(d,p) = (p-(n-1)/2)^2$ la borne inf\'erieure de ce support, et si $M$ est une $d$-vari\'et\'e hyperbolique on dira qu'une valeur propre $\lambda$ du laplacien $\Delta^p[M]$ est exceptionnelle si $\lambda < \lambda(d,p)$. Une telle valeur propre est n\'ec\'essairement isol\'ee. Le r\'esultat suivant est alors une cons\'equence de \cite[Theorem 1.9]{7samurai}. 

\begin{theo}
Pour touts $d,p,\, p<(d-1)/2$ et touts $\eps,\delta>0$ il existe une fonction d\'ecroissante $\alpha: [0, \lambda(d,p)]\to [0,1]$, telle que $\alpha(\lambda(d,p)) = 0$, ayant la propri\'et\'e suivante. Si $M$ est une $d$-vari\'et\'e hyperbolique compacte satisfaisant
$$
\vol(M_{\le \eps\log\vol(M)}) \le \vol(M)^{1-\delta}, \quad \inj(M) \ge \delta
$$
et $\lambda$ une valeur propre exceptionnelle de degr\'e $p$ de $M$ alors on a 
$$
m(\lambda,M) \le (\vol M)^{1-\alpha(\lambda)}
$$
pour $\vol(M)$ assez grand (d\'ependant de $\eps,\delta,d$). 
\label{vp_ramanujan}
\end{theo}


\subsection{Applications}
\label{applications}

\subsubsection{Rang sup\'erieur}

Soit $G$ un groupe de Lie r\'eel simple, de rang sup\'erieur et soit $X$ l'espace sym\'etrique associ\'e. Avec un peu de travail on peut d\'eduire du th\'eor\`eme de Nevo--St\"uck--Zimmer \ref{Nevo_Stuck_Zimmer} le r\'esultat suivant \cite[Theorem 1.5]{7samurai}. 

\begin{theo}
Soit $M_n$ une suite de $X$-vari\'et\'es (deux \`a deux non-isom\'etriques). Alors on a 
$$
\lim_{n\to+\infty} \frac{ \vol(M_n)_{\le R} }{ \vol M_n } = 0. 
$$
\label{BS_rangsup}
\end{theo}

On peut le reformuler de la mani\`ere suivante : il existe une fonction croissante $f$ sur $[0,+\infty[$ (d\'ependant de $X$) telle que $f(v)/v$ tende vers 0 quand $v\to+\infty$, et $\vol\left(M_{\le R}\right) \le f(\vol M)$ pour toute $X$-vari\'et\'e $M$. On peut de plus appliquer les r\'esultats de multiplicit\'es limites plus haut \`a la suite $M_n$ ou $\Gamma_n$. 


\subsubsection{Rev\^etements de congruence}
\label{rev_cong}
Si $\Gamma$ est un r\'eseau d'un groupe de Lie semisimple $G$ et $\Gamma_n$ une suite de sous-groupes d'indice fini sans torsion de $\Gamma$, alors la convergence de Benjamini--Schramm des vari\'et\'es $M_n=\Gamma_n\bs X$ est \'equivalente \`a ce que les $\Gamma_n$ satisfassent un crit\`ere introduit par Michael Farber dans le cadre d'une g\'en\'eralisation du th\'eor\`eme d'approximation de L\"uck (cf. \cite{Farber_paper}). 

\begin{lem}
Soient $\Gamma,\Gamma_n,M_n$ comme ci dessus, on suppose $G$ simple. Alors la suite $M_n$ est convergente au sens de Benjamini--Schramm vers $X$ si et seulement si, pour tout $g\in\Gamma$ semisimple la condition suivante est satisfaite:
\begin{equation}
\frac{|\{ \gamma\in \Gamma/\Gamma_n:\: \gamma^{-1} g\gamma\in\Gamma_n\}|} {|\Gamma/\Gamma_n|} \xrightarrow[n\to+\infty]{} 0. 
\label{Farber}
\end{equation}
\end{lem}

C'est en fait un cas particulier de la proposition \ref{geod_BSconv}, vu que le c\^ot\'e droit de \eqref{Farber} compte la proportion de la pr\'eimage de la g\'eod\'esique ferm\'ee associ\'ee \`a $\gamma$ qui soit de m\^eme longueur. 


La principale application de ce lemme est aux sous-groupes de congruence. Commen\c cons par rappeler ce que sont ces derniers : si $\Gamma$ est un r\'eseau arithm\'etique dans $G$ il existe un entier $n$ et une repr\'esentation fid\`ele $\rho : G\to \GL_n(\RR)$ telle que $\rho(\Gamma)$ soit contenu (forc\'ement avec indice fini) dans $G\cap\GL_n(\ZZ)$. Si $\Gamma'$ est un sous-groupe d'indice fini de $\Gamma$ on dit alors qu'il est de congruence s'il existe un entier $m\ge 1$ tel que $\Gamma'$ contienne le noyau du morphisme $\Gamma\to\GL(\ZZ/m)$ (compos\'ee de $\rho$ et de la r\'eduction modulo $m$ de $\GL_n(\ZZ)$)\footnote{Cette notion d\'epend a priori de la repr\'esentation $\rho$, mais uniquement \`a indice born\'e pr\`es. Pour une d\'efinition plus sophistiqu\'ee on ref\`ere \`a \cite[Chapter 6]{Lubotzky_Segal}.}. 

Dans ce cadre on a le r\'esultat suivant (cf. \cite[Theorem 1.12]{7samurai}). 

\begin{theo}
Si $\Gamma$ est un r\'eseau arithm\'etique cocompact de $G$ et $\Gamma_n$ est une suite de sous-groupes de congruence de $\Gamma$ (deux \`a deux distincts) alors il existe des constantes $c,\alpha>0$ telles que l'on ait pour touts $R>0$ et $n$ 
$$
\vol \left( (M_n)_{\le R}\right) \le e^{cR}(\vol M_n)^{1-\alpha}. 
$$
\label{BSconv_congruence}
\end{theo}

Dans le cas o\`u $\Gamma$ n'est que de covolume fini on n'obtient en g\'en\'eral que la convergence de Benjamini--Schramm des $M_n$ vers $X$ ; cependant il est probable que les m\^emes estim\'ees restent valides (cf. \cite[Theorem B]{moi2} pour le cas des vari\'et\'es hyperboliques en dimension 3). 

L'estim\'ee pr\'ecise sur le volume de la partie mince dans le th\'eor\`eme \ref{BSconv_congruence} permet aussi d'\^etre plus pr\'ecis pour les estim\'ees de multiplicit\'es de repr\'esentations non-temp\'er\'e via le th\'eor\`eme \ref{vp_ramanujan} . En particulier on a le r\'esultat suivant (cf. aussi \cite[Corollary 1.10]{7samurai}). 

\begin{theo}
Soient $d\ge 2$ et $M$ vari\'et\'e arithm\'etique hyperbolique de dimension $d$ ; pour tout $p\not= d/2$ (si $d$ est pair) ou $p\not=(d\pm 1)/2$ (si $d$ est impair) il existe un $\alpha>0$ tel que si $M_n$ une suite de rev\^etements de congruence de $M$ on ait 
$$
b_p(M_n) \ll (\vol M_n)^{1-\alpha}. 
$$
\end{theo}


\section{Vari\'et\'es hyperboliques r\'eelles}
\label{varietes}
\subsection{Sous-groupes al\'eatoires invariants de $\SO(n,1)$}
\label{IRS_hyperboliques}

\subsubsection{Sous-groupes normaux}

Les deux th\'eor\`emes \ref{Nevo_Stuck_Zimmer} et \ref{BS_rangsup} ci-dessus ne sont pas vrais en rang un. Plus pr\'ecis\'ement, on dispose de contre-exemples au th\'eor\`eme de Nevo--St\"uck--Zimmer pour tous les groupes de rang un, et au th\'eor\`eme \ref{BS_rangsup} au moins pour $G=\SO(n,1)$ ou $\SU(n,1)$. Les premiers proviennent de la construction d'IRS \`a partir de sous-groupes normaux de r\'eseaux de $G$ (cf. \ref{IRS_induit}) et du r\'esultat suivant, d\^u \`a Gromov et dont une preuve est donn\'ee par exemple dans \cite[Theorem 14.9]{Witte-Morris_livre}. 

\begin{theo}
Soit $G$ un groupe de Lie semisimple de rang r\'eel 1 et $\Gamma$ un r\'eseau de $G$. Il existe un sous-groupe normal $\Lambda\le \Gamma$ qui est infini et d'indice infini dans $\Gamma$. 
\end{theo}

En revanche la construction de $\Lambda$ dans le th\'eor\`eme ci-dessus (comme le sous-groupe normalement engendr\'e par une g\'eod\'esique ferm\'ee bien choisie) ne donne aucune indication sur la finitude r\'esiduelle du groupe quotient $\Gamma/\Lambda$ : on ne peut donc pas en d\'eduire un contre-exemple au second r\'esultat. 

Il est par contre connu que pour tout r\'eseau arithm\'etique $\Gamma$ de $\SO(n,1)$, $n\not= 7$ (\'egalement pour une grande partie des r\'eseaux arithm\'etiques en dimension 7 et pour tous les exemples non-arithm\'etiques connus), il existe un sous-groupe $\Gamma'\le \Gamma$ tel que $\Gamma'$ se surjecte sur $\ZZ$ (cf. \cite{Millson_Li} pour le cas arithm\'etique en dimensions $n\not=3,7$, \cite{Agol_VH} pour le cas g\'en\'eral en dimension 3 et \cite{Lubotzky_free_quotients} pour les vari\'et\'es hybrides). Le sous-groupe al\'eatoire invariant de $\SO(n,1)$ induit \`a partir du noyau de cette surjection donne alors un contre-exemple au th\'eor\`eme  \ref{BS_rangsup} pour $G=\SO(n,1)$. Il existe aussi des r\'eseaux dans $\SU(n,1)$ ayant un morphisme non-trivial vers $\ZZ$ (cf. \cite{Kazhdan_weilrepr}), et ceux-ci donnent donc aussi des contre-exemples pour $G = \SU(n,1)$.  

Noter que d'apr\`es le th\'eor\`eme \ref{BSconv_congruence}, si un r\'eseau arithm\'etique $\Gamma$ poss\`ede un IRS support\'e sur des sous-groupes infinis, d'indice infini qui soit limite de sous-groupes d'indice fini (en particulier s'il existe un sous-groupe normal non-trivial $\Lambda\le\Gamma$ tel que $\Gamma/\Lambda$ soit infini et r\'esiduellement fini) alors le noyau de congruence de $\Gamma$ est infini (en particulier, le th\'eor\`eme \ref{Nevo_Stuck_Zimmer} est compatible avec la propri\'et\'e des sous-groupes de congruence pour les r\'eseaux d'ordre sup\'erieur). Dans le cas hyperbolique complexe, il existe en toute dimension des vari\'et\'es arithm\'etiques pour lesquelles on ne conna\^it pas l'existence d'un sous-groupe d'indice fini ayant une ab\'elianisation infinie (et dont on sait en fait, d'apr\`es un th\'eor\`eme de Jon Rogawski \cite{Rogawski_U3}, que tous leurs sous-groupes de congruence ont un premier nombre de Betti nul). Pour ces derniers il serait donc particuli\`erement int\'eressant d'exhiber un IRS approximable (ou un sous-groupe $\Lambda$ comme ci-dessus). 


\subsubsection{Classifications topologiques en petites dimensions}

On peut obtenir une classifiction partielle des vari\'et\'es apparaissant dans les sous-groupes al\'eatoires invariants des isom\'etries du plan ou de l'espace hyperbolique. Pour ce qui est des surfaces on a une classification topologique compl\`ete, pour les vari\'et\'es de dimension trois on doit se limiter aux types topologiques finis. Le r\'esultat pour ces derni\`eres est d\^u \`a Mikl\'os Ab\'ert et Ian Biringer. Il est d\'emontr\'e dans \cite{Abert_Biringer} et s'\'enonce comme suit.

\begin{theo}[Ab\'ert--Biringer]
Soit $\mu$ un IRS ergodique de $\PSL_2(\CC)$ tel que $\mu$-presque tout sous-groupe soit finiment engendr\'e. On suppose de plus que $\mu$ n'est pas support\'e sur le sous-groupe trivial ou sur des r\'eseaux. Alors $\mu$ est support\'e sur des sous-groupes de surfaces doublement d\'eg\'en\'er\'es ; en particulier il existe une surface hyperbolique $S$ telle que pour $\mu$-presque tout $\Gamma$ la vari\'et\'e $\Gamma\bs\HH^3$ soit diff\'eomorphe \`a $S\times\RR$. 
\label{fgIRS_dim3}
\end{theo}

Le moyen le plus simple de construire de tels IRS est de consid\'erer une vari\'et\'e hyperbolique fibr\'ee $M = S\times[0,1]/(x,0)\sim(\phi x,1)$ (o\`u $\phi$ est un diff\'eomorphisme pseudo-Anosov de la surface hyperbolique $S$) qui poss\`ede alors un rev\^etement cyclique infini $S\times\RR$ dont la monodromie est un groupe de surface doublement d\'eg\'en\'er\'e. Il existe aussi des exemples nettement plus compliqu\'es donnant des groupes qui ne sont pas contenus dans des r\'eseaux de $\PSL_2(\CC)$ (qui apparaissent d'ailleurs comme limites des exemples pr\'ec\'edents). Ces derniers sont construits dans \cite[Theorem 12.8]{7samurai}. 

Il existe, en toute dimension, des sous-groupes al\'eatoires invariants support\'es sur des groupes de rang infini. Un exemple simple est donn\'e par les rev\^etements cycliques infinis de vari\'et\'es de volume fini. Des exemples plus complexes (contenant en particulier une quantit\'e non-d\'enombrable de types topologiques) seront pr\'esent\'es ci-dessous (cf. \ref{IRS_hybrides}).

\medskip

On peut compl\`etement caract\'eriser les types topologiques des surfaces apparaissant dans des IRS de $\PSL_2(\RR)$ par le r\'esultat suivant, dont une preuve est donn\'ee dans l'appendice \ref{IRS_dim2} (\'ecrit avec I. Biringer). 

\begin{theo}
Soit $\mu$ un IRS ergodique, sans atomes de $\PSL_2(\RR)$. Alors il existe une surface $S_\mu$ telle que pour $\mu$-presque tout $\Lambda$ la surface $\Lambda\bs\HH^2$ soit hom\'eomorphe \`a $S_\mu$. 

De plus si $S_\mu$ n'est pas de type topologique fini alors elle est hom\'eomorphe \`a l'une des dix surfaces suivantes : le monstre du Loch Ness (plan auquel on a attach\'e une infinit\'e d'anses), l'\'echelle de Jacob (double du pr\'ec\'edent moins un disque), l'arbre de Cantor (sph\`ere priv\'ee d'un sous-ensemble de Cantor) ou l'arbre de Cantor fleuri (le dernier avec une infinit\'e d'anses attach\'ees, chaque point du Cantor \'etant limite d'anses)---cf. la figure \ref{surfaces} pour les repr\'esentations standard de ces derni\`eres---ou l'une de celles-ci, le cylindre ou le plan \`a laquelle on a \^ot\'e un sous-ensemble localement fini de points qui intersecte tous les voisinages de bouts. 
\label{IRS_dim2_theo}
\end{theo}

\begin{figure}
\centering 

\begin{subfigure}[b]{0.4\textwidth}
\includegraphics[width = \textwidth]{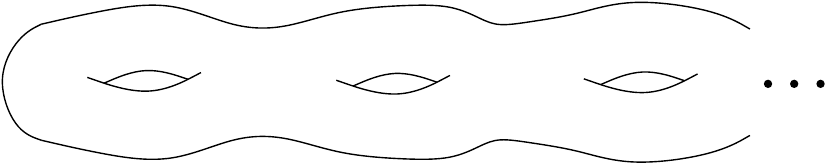}
\caption{Monstre du Loch Ness}
\end{subfigure}
\hspace{1cm}
\begin{subfigure}[b]{0.4\textwidth}
\includegraphics[width = \textwidth]{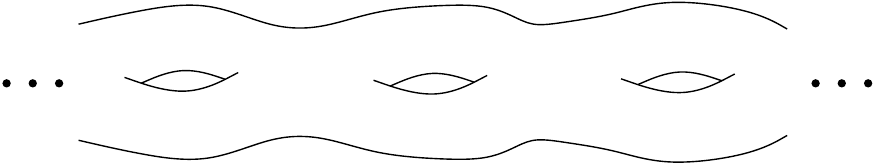}
\caption{Echelle de Jacob}
\end{subfigure}

\vspace{1cm}

\begin{subfigure}[b]{0.4\textwidth}
\includegraphics[width = \textwidth]{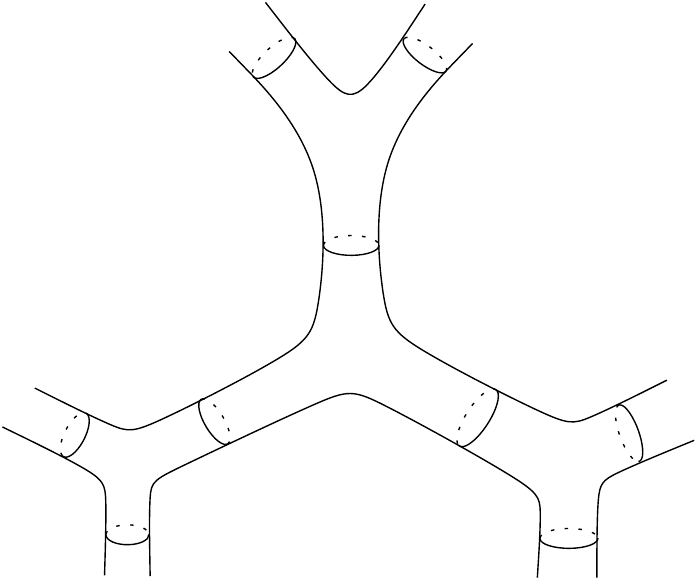}
\caption{Arbre de Cantor}
\end{subfigure}
\hspace{1cm}
\begin{subfigure}[b]{0.4\textwidth}
\includegraphics[width = \textwidth]{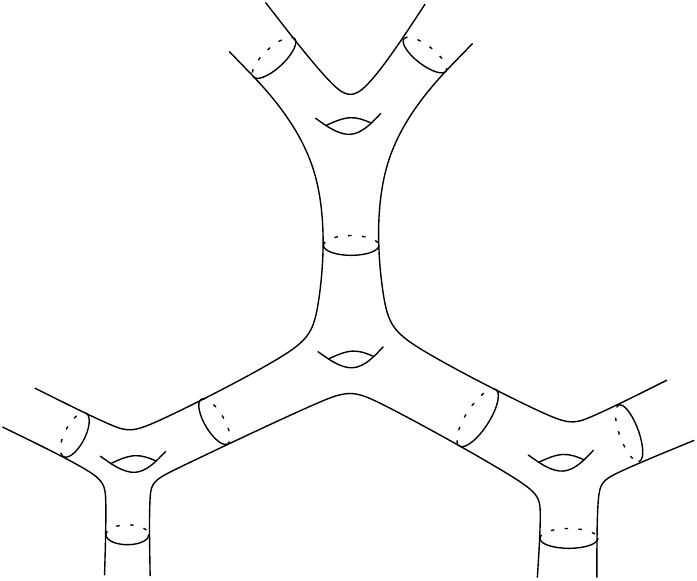}
\caption{Arbre de Cantor fleuri}
\end{subfigure}

\caption{}\label{surfaces}
\end{figure}

Il n'est pas dur de construire pour chacun de ces types un IRS de $\PSL_2(\RR)$ qui soit support\'e sur des surfaces de ce type. On peut le faire par des rev\^etements de surfaces de volume fini\footnote{Par exemple le monstre du Loch Ness, l'\'echelle de Jacob et les arbres de Cantor correspondent respectivement \`a des rev\^etements ab\'eliens libres de rang 1, $\ge 2$ et libres non-ab\'eliens d'une surface compacte que l'on peut remplacer par une surface \`a cusps pour obtenir les types restants.}, ou par des d\'ecompositions en pantalons (\'eventuellement d\'eg\'en\'er\'es) sur des graphes infinis (cf. \cite[12.1]{7samurai} pour cette derni\`ere construction). On remarque que la g\'eom\'etrie des rev\^etements infinis de surfaces a \'et\'e \'etudi\'e par Rostislav Grigorchuk dans \cite{Grigorchuk_surfaces} (dans ce cas le th\'eor\`eme \ref{IRS_dim2_theo} est une cons\'equence imm\'ediate de la classification des surfaces (rappel\'ee dans l'appendice \ref{IRS_dim2}) et de l'observation de Heinz Hopf que les groupes infinis ont un, deux ou un ensemble de Cantor de bouts). 

Enfin, notons qu'au vu des liens entre les sous-groupes al\'eatoires invariants et les mesures harmoniques des feuilletages (cf. \cite{Abert_Biringer}) ce r\'esultat peut \^etre vu comme un analogue dans le premier cadre au th\'eor\`eme d'Etienne Ghys sur les feuilles g\'en\'eriques \cite{Ghys_feuilles}. 


\subsubsection{Dimensions sup\'erieures}
\label{IRS_hybrides}

Comme pour la construction de vari\'et\'es non-arithm\'etiques, en dimensions plus grandes on ne dispose d'aucune approche syst\'ematique. On peut cependant construire des exemples d'IRS ergodiques `exotiques' (qui ne sont pas induits par un IRS d'un r\'eseau) en toute dimension comme suit ; les d\'etails sont donn\'es dans \cite[Section 13]{7samurai}. 

Soit $d\ge 3$ ; on choisit deux $d$-vari\'et\'es hyperboliques compactes $N_0,N_1$ ayant chacune un bord totalement g\'eod\'esique, compos\'e de deux copies d'une vari\'et\'e hyperbolique $\Sigma$ de dimension $d-1$. Pour une suite $\alpha\in\{0,1\}^{\mathbb{Z}}$ on note $N_{\alpha}$ la vari\'et\'e hyperbolique compl\`ete de volume infini obtenue en recollant des copies de $N_0,N_1$ de la mani\`ere indiqu\'ee par $\alpha$. Plus pr\'ecis\'ement, on suppose que l'on a choisi une identification des bords de $N_0$ et $N_1$, et on note $i_a^\pm$ les inclusions $\sigma \to N_a$ ; on a alors :
$$
N_{\alpha} = \left( \bigsqcup_{i\in\ZZ} N_{\alpha_i} \times \{i\} \right)/ (i_{\alpha_i}^+ x,i) \sim (i_{\alpha_{i+1}}^- x,i+1)\quad (i\in\ZZ, \,x\in\Sigma). 
$$
Si on prend maintenant n'importe quelle mesure de probabilit\'e $\nu$ sur $\{0,1\}^{\mathbb{Z}}$ on obtient un sous-groupe al\'eatoire $\mu_\nu$ de $\SO(d,1)$, obtenu en choisissant selon la mesure $G$-invariante un rep\`ere al\'eatoire dans $N_{\alpha_0}\times\{0\}\subset N_\alpha$ o\`u $\alpha$ est choisie al\'eatoirement suivant la loi $\nu'$ sur $\{0,1\}^\ZZ$ d\'efinie par :
$$
\nu'(A)=\frac{\int_A\vol(N_{\alpha_0})d\nu(\alpha)}{\int_{\{0,1\}^{\mathbb{Z}}}\vol(N_{\alpha_0})d\nu(\alpha)}.
$$
Si la mesure originale $\nu$ est invariante par le d\'ecalage alors $\mu_\nu$ est invariant par conjugaison, et si de plus $\nu$ est ergodique alors $\mu_\nu$ l'est aussi. Dans le cas o\`u $\nu$ est la mesure invariante support\'ee sur l'orbite par d\'ecalage d'une suite p\'eriodique $\alpha$ il est clair que $\mu_\nu$ est l'IRS induit correspondant au rev\^etement cyclique infini de $N_\alpha$ sur la vari\'et\'e compacte obtenue en quotientant par la puissance du d\'ecalage correspondant \`a la p\'eriode. Dans les cas restants on obtient bien de nouveaux IRS, comme montr\'e par le r\'esultat suivant. 

\begin{theo}
Soit $n\ge 3$, il existe alors un choix de $N_0$ et $N_1$ telles que pour une suite $\alpha\in\{0,1\}^\ZZ$ non-p\'eriodique la vari\'et\'e $N_{\alpha}$ ne soit un rev\^etement d'aucune vari\'et\'e de volume fini. En particulier, si la mesure invariante ergodique $\nu$ n'est pas support\'ee sur une orbite p\'eriodique alors l'IRS ergodique $\mu_\nu$ ne peut pas \^etre induit depuis un r\'eseau de $\SO(d,1)$. 
\end{theo}

Le choix de $N_0,N_1$ utilis\'e dans la preuve de ce th\'eor\`eme est inspir\'e par la construction de vari\'et\'es non-arithm\'etiques par Gromov et Piatetski-Shapiro \cite{GPS}. 



\subsection{Surfaces al\'eatoires}

\label{random_surfaces}

\subsubsection{Weil--Petersson}

L'espace de modules ${\mathcal M}_g$ des structures hyperboliques \`a isom\'etrie pr\`es sur une surface de genre $g$ est muni d'une mesure Bor\'elienne $\nu_{\rm wp}$, dite de Weil--Petersson. La masse totale est finie, et on notera $\mu_{\rm wp}$ la mesure de probabilit\'e associ\'ee. Le r\'esultat suivant est d\'emontr\'e par Maryam Mirzakhani dans \cite[4.4]{Mirzakhani}. 

\begin{theo}[Mirzakhani]
Il existe une constante $C>1$ telle que pour tout $g\ge 2$ on ait:
$$
\mu_{\rm wp} \left( X\in{\mathcal M}_g: \: \vol(X_{\le \log(g)/6}) \le C^{-1}g^{11/12}\log(g) \right) \ge 1 - Cg^{-1/4}.
$$
\label{WP_conv}
\end{theo}


\subsubsection{Surfaces de B\'elyi al\'eatoires}
\label{belyi_maintext}

Un autre exemple int\'eressant de mod\`ele al\'eatoire est d\'efini et \'etudi\'e par Robert Brooks et Eran Makover dans \cite{Brooks_Makover}. Une surface de B\'elyi est la compactification conforme d'une surface arithm\'etique non-compacte : on part d'un rev\^etement $S$ de la surface modulaire $\PSL_2(\ZZ)\bs\HH^2$, que l'on munit de sa structure conforme. La surface non-compacte $S$ est diff\'eomorphe \`a une surface compacte $S_C$ \`a laquelle on \^ote un nombre fini de points ; de plus la structure conforme sur $S$ induit une structure conforme sur $S_C$, et cette derni\`ere est alors appel\'ee une surface de B\'elyi. 

Il est facile de construire al\'eatoirement des surfaces arithm\'etiques non-compactes : si $\mathcal G$ est un graphe trivalent on otient une telle surface en recollant des triangles hyperboliques id\'eaux selon le sch\'ema prescrit par $\mathcal G$, en identifiant les c\^ot\'es sans param\`etre de cisaillement (i.e. les projections des centres de gravit\'es de deux triangles adjacents \`a un c\^ot\'e sur ce dernier co\"incident)\footnote{Pour lever certaines ambigu\"it\'es il faut aussi munir $\mathcal G$ d'un ordre cyclique des ar\^etes adjacentes \`a chacun de ses sommets, on ref\`ere \`a \cite{Brooks_Makover} pour plus de d\'etails.}. N'importe quel mod\`ele de graphe al\'eatoire donne alors un mod\`ele al\'eatoire pour les surfaces arithm\'etiques, et partant pour les surfaces de B\'elyi obtenues en compactifiant ces derni\`eres. 

Le mod\`ele que l'on retiendra est le m\^eme que dans \cite{Brooks_Makover}. Une surface de B\'elyi al\'eatoire de complexit\'e $n$ est obtenue en tirant au hasard un graphe trivalent \`a $2n$ sommets comme suit : on choisit, selon la loi uniforme, une bijection $\{1,\ldots,6n\} \to \{1,\ldots,2n\}\times\{1,2,3\}$, $i\mapsto(a_i,b_i)$ et on met une ar\^ete entre les sommets $a_{2i-1},a_{2i}$ pour $i=1,\ldots,3n$. Un th\'eor\`eme de B\'ela Bollob\'as (cf. \cite[Theorem 5.3]{Brooks_Makover}) montre que pour $r,m$ donn\'es, la nombre de circuits de longueur $r$ dans $\mathcal G$ est inf\'erieur \`a $m$ avec probabilit\'e tendant vers 1 quand $m$ tend vers l'infini (ind\'ependamment de $n$ assez grand). Par un r\'esultat de comparaison d\^u \`a Brooks on peut en d\'eduire le r\'esultat suivant (la preuve compl\`ete est donn\'ee dans l'annexe \ref{belyi}). 

\begin{theo}
Pour touts $R,\eps>0$, la probabilit\'e que pour une surface de B\'elyi al\'eatoire $S_C$ de complexit\'e $n$ le volume de la partie $R$-mince $(S_C)_{\le R}$ soit sup\'erieur \`a $\eps\vol(S)$ tend vers z\'ero quand $n$ tend vers l'infini. 
\label{Belyi_conv}
\end{theo}

Noter que l'esp\'erance de la systole de $S_C$ est born\'ee (cf. \cite{Petri_systole} qui calcule la limite exacte) et que celle du rayon maximal est d'ordre $\log\vol(S_C)$ (cf. \cite[Theorem 2.4]{Brooks_Makover}). 


\subsubsection{Application aux petites valeurs propres}

Nous commen\c cons par les surfaces al\'eatoires d\'etermin\'ees par la loi de Weil--Petresson. Le th\'eor\`eme de Mirzakhani \ref{WP_conv} conjointement aux r\'esultats de \cite{7samurai} (rappel\'es en \ref{multlim} ci-dessus) montrent qu'une surface al\'eatoire de Weil--Petersson g\'en\'erique a peu de petites (i.e. inf\'erieures \`a $1/4$) valeurs propres du Laplacien sur les fonctions. Rappelons que $m(S,\lambda)$ d\'esigne la multiplicit\'e de $\lambda\in[0,+\infty[$ comme valeur propre du Laplacien sur les fonctions de carr\'e int\'egrable sur $S$. On a alors les deux r\'esultats suivants. 
\begin{itemize}
\item[(i)] Il existe une fonction $f$ telle que $f(g)/g\xrightarrow[g\to+\infty]{} 0$ et pour une surface al\'eatoire de Weil--Petersson de genre $g$ on ait avec probabilit\'e tendant vers 1 quand $g\to+\infty$
$$
\sum_{\lambda<1/4} m(\lambda,S) \le f(g). 
$$
\item[(ii)] Il existe un $C>0$ et une fonction d\'ecroissante $\alpha:[0,1/4]\to [0,1[$ (avec $\alpha(1/4)=0$) tels que pour tout $\lambda\in [0,1/4[$ et pour une surface al\'eatoire de Weil--Petersson de genre $g$ on ait avec probabilit\'e tendant vers 1 quand $g\to+\infty$
$$
m(\lambda,S) \le Cg^{1-\alpha(\lambda)}. 
$$
\end{itemize}

Pour les surfaces de B\'elyi al\'eatoires on obtient par le th\'eor\`eme \ref{Belyi_conv} l'analogue de (i), en revanche l'\'etude des parties minces n'est pas assez fine pour obtenir un r\'esultat \'equivalent \`a (ii). 

Jean-Pierre Otal et Eulalio Rosas ont montr\'e dans \cite{Otal_Rosas} que pour une surface de genre $g$ il y a au plus $2g-1$ petites valeurs propres, et que cette borne est optimale si on consid\`ere toutes les surfaces de genre $g$. En revanche, le point (i) ci-dessus montre que la probabilit\'e (pour la mesure de Weil--Petersson normalis\'ee, ou pour le mod\`ele de Brooks--Makover) qu'une surfaces de genre $g$ ait autant de valeurs propres tend vers 0 quand $g$ augmente. Le point (ii) est optimal (\`a l'\'echelle logarithmique) pour l'ensemble de toutes les surfaces : Bruno Colbois et Yves Colin de Verdi\`ere ont construit dans \cite{CCdV} des surfaces hyperboliques de genre $g$ dont la premi\`ere valeur propre est $<1/4$ et de multiplicit\'e $\gg \sqrt g$. 


\subsection{Vari\'et\'es hyperboliques de grand volume en dimensions $\ge 3$}
\label{higher_dim}

\subsubsection{Vari\'et\'es arithm\'etiques}

Un r\'eseau arithm\'etique dans $\SO(d,1)$ est d\'efini, \`a commensurabilit\'e pr\`es, par un corps de nombres totalement r\'eel $F$ et un groupe alg\'ebrique $\G/F$ tel que pour tous les plongements $\sigma:F\to\RR$ sauf un on ait $\G^\sigma(\RR) = \SO(d+1)$, et pour le plongement $\sigma_0$ restant $\G^{\sigma_0}(\RR) = \SO(d,1)$. Le groupe des points entiers $\G(\so_F)$ (d\'efini \`a commensurabilit\'e pr\`es) est alors un r\'eseau de $\SO(d,1)\times\SO(d+1)^{r-1}$ (o\`u $r = [F:\QQ]$) et sa projection $\Gamma$ dans $G = \SO(d,1)$ est un r\'eseau dans ce dernier groupe. On dira que $F$ est le corps de d\'efinition de $\Gamma$ (ou de n'importe quel sous-groupe de $G$ qui lui est commensurable). La conjecture suivante appara\^it dans \cite{moi1}, et il semble que Peter Sarnak en ait \'enonc\'e une semblable. 

\begin{conj}
Soit $\Gamma_n$ une suite de r\'eseaux sans torsion maximaux (deux \`a deux non-conjugu\'es) dans $\SO(d,1)$. Alors la suite des orbifolds $\Gamma_n\bs\HH^d$ est convergente au sens de Benjamini--Schramm vers $\HH^d$. 
\label{arit_conv}
\end{conj}

Une justification possible pour cette conjecture est donn\'ee par un r\'esultat de M. Ab\'ert et B. Vir\'ag (cf. \cite{Miklos_slides}), qui d\'eduit la convergence de l'existence d'un trou spectral optimal (i.e. de la conjecture de Ramanujan). En toute g\'en\'eralit\'e cette conjecture appara\^it comme difficile \`a d\'emontrer (si seulement elle est vraie!), mais pour les petites dimensions on peut l'attaquer directement. En effet, pour $d=2,3$ on a une description assez maniable des groupes $\G$ permettant la construction de r\'eseaux arithm\'etiques ; celle-ci provient en dimension 2 de l'isomorphisme exceptionnel $\SO(2,1)\cong\PSL_2(\RR)$ (resp. $\SO(3,1)\cong\PSL_2(\CC)$ en dimension 3). Un r\'eseau arithm\'etique $\Gamma$ de $\PSL_2(\CC)$ est d\'ecrit \`a commensurabilit\'e pr\`es par un corps de nombres $F$ ayant exactement une place complexe, et une alg\`ebre de quaternions $A$ sur $F$ ramifi\'ee \`a toutes les places r\'eelles de $F$. Le corps $F$ est alors appel\'e corps des traces invariants de $\Gamma$ (et peut effectivement \^etre obtenu \`a partir des traces des carr\'es des \'el\'ements de $\Gamma$). Les longueurs des g\'eod\'esiques ferm\'ees du quotient $\Gamma\bs\HH^3$ sont alors d\'ecrites par certaines extensions quadratiques de $F$. On ref\`ere \`a \cite{MR} et \cite{moi1} pour une description plus pr\'ecise de cette construction de r\'eseaux arithm\'etiques. 

Le r\'esultat suivant est en grande partie issu de \cite{moi1}. 

\begin{theo}
Soit $p$ un nombre premier ; si $\Gamma_n$ est une suite de r\'eseaux arithm\'etiques sans torsion maximaux de $\PSL_2(\CC)$ dont les corps des traces invariants sont tous de degr\'e $p$ alors la suite des vari\'et\'es $\Gamma_n\bs\HH^3$ est convergente au sens de Benjamini--Schramm vers $\HH^3$. 
\label{kleinian_arit_conv}
\end{theo}

Dans le cas o\`u le corps des traces invariants est fix\'e ce r\'esultat est une cons\'equence de la description des r\'eseaux maximaux et de la correspondance de Jacquet--Langlands et on n'a en fait pas besoin de l'hypoth\`ese sur le degr\'e du corps de traces (cf. la preuve de la proposition 6.7 dans \cite{moi1}). Dans le cas o\`u les corps de traces sont deux \`a deux distincts on utilise le lemme suivant. 

\begin{lem}
Soient $p$ un nombre premier et $u$ une unit\'e alg\'ebrique qui ne soit ni quadratique r\'eelle, ni une racine de l'unit\'e. Il n'existe qu'un nombre fini de corps de nombres $F$ de degr\'e $p$, ayant une seule place complexe et tels que $E=F(u)$ soit une extension quadratique de $F$ v\'erifiant $|u|_{E/F} = 1$. 
\label{fini_corps}
\end{lem} 

Via la correspondance entre g\'eod\'esiques et extensions quadratiques on en d\'eduit que si $p>2$ alors on obtient en fait par le lemme \ref{fini_corps} que $\inj(\Gamma_n\bs\HH^3)$ tend vers l'infini quand le discriminant de $F$ tend vers l'infini. Dans le cas restant o\`u $p=2$ on utilise le fait que par le th\'eor\`eme \ref{Borel_density} un groupe dans le support d'un IRS limite est soit trivial, soit Zariski dense ; comme dans le dernier cas il doit contenir un \'el\'ement dont les valeurs propres ne sont pas totalement r\'eelles, ce cas est impossible par le lemme \ref{fini_corps}. 

Une preuve diff\'erente, reposant sur une description plus pr\'ecise de la g\'eom\'etrie des groupes $\SL_2(\so_F)$ (utilisant en particulier un r\'esultat de Shin Ohno et Takao Watanabe \cite{Ohno_Watanabe})  et la correspondance de Jacquet--Langlands est aussi donn\'ee dans \cite{moi1}. Elle ne fonctionne cependant que pour les degr\'es 2 et 3. 


\subsubsection{Topologie de la dimension trois}
\label{topo_dim3}

Rappelons les d\'efinitions de quelques invariants topologiques des vari\'et\'es de dimension trois :
\begin{itemize}
\item Si $H$ est un corps \`a anses de genre $g$ (i.e. un voisinage r\'egulier dans $\RR^3$ d'un bouquet de $g$ cercles plong\'e) et $\phi$ une classe d'hom\'eomorphisme de $\pl H$ alors $H\cup_\phi H$ (o\`u l'on identifie $x\in\pl H$ dans la premi\`ere copie de $H$ avec $\phi(x)$ dans la seconde) est une vari\'et\'e compacte sans bord de dimension 3 ; toute telle vari\'et\'e $M$ admet une telle d\'ecomposition (`scindement de Heegard') et on d\'efinit le genre de Heegard de $M$ comme le plus petit $g$ telle que $M$ ait un scindement en deux corps \`a anses de genre $g$. 
\item Si $S$ est une surface et $\phi$ une classe d'hom\'eomorphisme de $S$ alors la suspension $S\times [0,1]/\sim$ (o\`u l'on fait l'identification $(x,0)\sim(\phi(x),1)$) est une vari\'et\'e de dimension trois, que l'on appelle fibr\'ee de fibre $S$. Beaucoup de vari\'et\'es hyperboliques ne sont pas ainsi fibr\'ees, mais c'est un th\'eor\`eme r\'ecent de Ian Agol \cite{Agol_VH} que tout vari\'et\'e hyperbolique compacte a un rev\^etement fini qui fibre. Si $S$ est une surface immerg\'ee dans $M$ qui se rel\`eve dans un rev\^etement $M'$ en une fibre $S'$ d'une fibration de $M'$, on l'appelle fibre virtuelle dans $M$. 
\end{itemize}

Ces propri\'et\'es ont des lien bien connus avec la g\'eom\'etrie, en particulier si une vari\'et\'e hyperbolique $M$ admet un scindement de Heegard de genre $g$ alors son rayon, d'injectivit\'e global $\inj(M)$ est born\'e par une constante ne d\'ependant que de $g$ ; de m\^eme, si $M$ est fibr\'ee avec une fibre de genre $g$ on a une borne sur $\inj(M)$ que ne d\'epend que de $g$. On obtient ainsi le r\'esultat suivant. 

\begin{prop}
Si $M_n$ est une suite de vari\'et\'es hyperboliques compactes de dimension 3 qui soit convergente au sens de Benjamini--Schramm vers $\HH^3$ alors le genre de Heegard de $M_n$ et le genre minimal d'une fibre virtuelle dans $M_n$ tendent tous deux vers l'infini. 
\end{prop} 

En particulier, avec la th\'eor\`eme \ref{kleinian_arit_conv} on obtient le r\'esultat suivant.

\begin{cor}
Si $g$ est fix\'e $p$ est un nombre premier alors il n'existe qu'un nombre fini de r\'eseaux arithm\'etiques maximaux (resp. de r\'eseaux de congruence, ou de classes de commensurabilit\'es de r\'eseaux arithm\'etiques) $\Gamma$ dont le corps de traces invariant est de degr\'e $p$ et tels que le genre de Heegard de $\Gamma\bs\HH^3$ (ou le genre minimal d'une fibre virtuelle dans $\Gamma\bs\HH^3$) soit \'egal \`a $g$.
\end{cor}


\subsubsection{Vari\'et\'es non-arithm\'etiques en dimensions sup\'erieures}

Soit $d\ge 4$ ; c'est une cons\'equence bien connue de la rigidit\'e locale des r\'eseaux de $\SO(d,1)$ (originellement observ\'ee par Hsien-Chung Wang) qu'\'etant donn\'e un $v>0$ il n'existe qu'un nombre fini de vari\'et\'es hyperboliques de dimension $n$ et de volume plus petit que $v$. On peut noter ${\mathcal M}_d(v)$ cet ensemble, on a alors 
$$
\log m(v) := \log|{\mathcal M}_d(v)| \asymp v\log(v). 
$$
Au vu des r\'esultats en dimension 2 et en rang sup\'erieur il est assez naturel de se poser, \`a la suite de Shmuel Weinberger (cf. \cite{Miklos_slides}), la question suivante. 

\begin{question}
Etant donn\'es $R,\eps>0$, soit $p(v)$ la proportion de $M\in {\mathcal M}_d(v)$ telles que $\vol (M_{\le R}) \le \eps\vol M$. Est-ce que $p(v)/m(v)$ tend vers z\'ero quand $v$ grandit?
\label{why}
\end{question}

La construction de vari\'et\'es non-arithm\'etiques par Gromov et Piatetski-Shapiro \cite{GPS} a \'et\'e reprise dans \cite{moi3} puis par Tsachik Gelander et Arie Levit dans \cite{Gelander_Levit}. On obtient ainsi des suites de groupes maximaux sans torsion telles que les vari\'et\'es associ\'ees aient un rayon maximal born\'e, et donc ne soient pas BS-convergentes vers le rev\^etement universel (leurs limites dans la topologie de Benjamini--Schramm sont les IRS construits en \ref{IRS_hybrides}). Le r\'esultat principal dans ce dernier article montre que pour $R$ assez grand on a $\log p(v) \asymp v\log(v)$, ce qui inciterait \`a penser que la r\'eponse \`a la question \ref{why} pourrait \^etre n\'egative. Cependant notre m\'econnaissance quasi-compl\`ete du paysage global de l'ensemble des vari\'et\'es hyperboliques en grande dimension ne saurait permettre la moindre certitude quand \`a ce sujet.


\appendix

\section[Topologie des surfaces al\'eatoires unimodulaires]{Topologie des surfaces al\'eatoires unimodulaires \\ (par I. Biringer et J. Raimbault)}
\label{IRS_dim2}
On donne ici la d\'emonstration du th\'eor\`eme \ref{IRS_dim2_theo}. 

\subsection{Classification topologique des surfaces}

Commen\c cons par rappeler la classification topologique des surfaces de type infini. On rappelle que l'espace ${\mathcal E}(S)$ des bouts de $S$ est d\'efini comme la limite inductive des compl\'ementaires des sous-surfaces compactes de $S$\footnote{En termes plus intelligibles, \'etant donn\'ee une suite $K_n$ de sous-surfaces compactes de $S$ avec $K_n\subset K_{n+1}$ et $\bigcup_n K_n = S$, un bout de $S$ est d\'etermin\'e par une suite $E_n$, o\`u chaque $E_n$ est une composante connexe (forc\'ement non-born\'ee) du compl\'ementaire de $K_n$ avec $E_{n+1}\subset E_n$.}. C'est un espace compact totalement discontinu. Le {\it genre} (possiblement infini) d'une surface $S$ est le nombre maximum de courbes simples non s\'eparantes, deux \`a deux disjointes et non isotopes, sur $S$. Le genre d'un bout $E$ de $S$ est le minimum des genres des sous-surfaces $S'$ de $S$ contenant $E$; il ne peut \^etre \'egal qu'\`a z\'ero ou \`a l'infini. Le th\'eor\`eme suivant est d\^u \`a B\'ela Ker\'ekj\'art\'o, une d\'emonstration compl\`ete en est donn\'ee par Ian Richards dans \cite{Richards_surfaces}. 

\begin{theo}[Ker\'ekj\'art\'o, Richards]
Deux surfaces topologiques orientables $S$ et $S'$ sont hom\'eomorphes si et seulement si elles ont m\^eme genre, et s'il existe un hom\'eomorphisme ${\mathcal E}(S) \to {\mathcal E}(S')$ respectant les genres. 
\label{class}
\end{theo}

Un bout est dit cuspidal s'il est de genre 0 et isol\'e dans ${\mathcal E}(S)$ ; vu que dans une surface hyperbolique un tel bout est un cusp ou un entonnoir et que le second cas est impossible par le th\'eor\`eme \ref{limit_set} on a le lemme suivant, qui justifie cette appellation. 

\begin{lem}
Si $\nu$ est un IRS sans atome de $\PSL_2(\RR)$ alors tout bout isol\'e de genre 0 de $S$ est un cusp. 
\end{lem}

On notera ${\mathcal E}_{\mathrm nc}(S)$ l'ensemble des bouts non-cuspidaux de $S$, qui est un sous-espace ferm\'e du compact ${\mathcal E}(S)$. Le r\'esultat que nous d\'emontrerons plus loin est le suivant. 

\begin{theo}
Soit $\nu$ un sous-groupe al\'eatoire invariant de $\PSL_2(\RR)$. Alors pour $\nu$-presque tout $\Gamma$, si la surface $S=\Gamma\bs\HH^2$ n'est pas de volume fini elle v\'erifie les propri\'et\'es suivantes :
\begin{itemize}
\item[(i)] Si $|{\mathcal E}_{\mathrm nc}(S)|>2$ alors ${\mathcal E}_{\mathrm nc}(S)$ est un ensemble de Cantor ;  
\item[(ii)] Si $S$ a un bout de genre 0 alors $S$ elle-m\^eme est de genre 0 ;  
\item[(iii)] Si $S$ a un bout cuspidal alors les bouts cuspidaux sont denses dans ${\mathcal E}(S)$. 
\end{itemize}
\label{theo_appendix}
\end{theo}

Montrons comment on en d\'eduit le th\'eor\`eme \ref{IRS_dim2_theo} : soit $\nu$ un IRS ergodique de $\PSL_2(\RR)$. Les conditions (i), (ii) et (iii) ci-dessus d\'eterminent d'apr\`es la classification (theor\`eme \ref{class}) exactement douze types topologiques ; en particulier (comme il n'y a qu'une quantit\'e d\'enombrable de types topologiques finis) le type topologique de $\Gamma\bs\HH^2$ appartient presque s\^urement \`a un ensemble d\'enombrable. Vu que l'ensemble des $\Gamma$ tels que $\Gamma\bs\HH^2$ ait un type topologique donn\'e est un bor\'elien invariant par $\PSL_2(\RR)$ il suit alors de l'ergodicit\'e de $\nu$ qu'une seule surface peut \^etre r\'ealis\'ee comme $\Gamma\bs\HH^2$ avec une probabilit\'e non nulle. Il reste \`a voir que les types infinis sont bien ceux indiqu\'es dans l'\'enonc\'e : deux des surfaces v\'erifiant (i), (ii) et (iii) c-dessus sont le plan et le cylindre, qui n'apparaissent pas dans le support d'IRS non-atomiques de $\PSL_2(\RR)$. On laisse au lecteur le soin de v\'erifier que les dix surfaces restantes sont bien celle apparaissant dans le th\'eor\`eme \ref{IRS_dim2_theo}. 


\subsection{Transport de masse}

Le contenu de cette section vient de \cite{Biringer_Tamuz} (voir aussi \cite{Abert_Biringer}). Si $\nu$ est un IRS de $\PSL_2(\RR)$ sans atome, on note $\mu_\nu$ la mesure pouss\'ee en avant sur l'espace $\mathcal S$ des surfaces hyperboliques point\'ees. On note ${\mathcal S}_2$ l'espace des surfaces hyperboliques doublement point\'ees. Si $f$ est une fonction bor\'elienne sur ${\mathcal S}_2$ l'\'egalit\'e suivante est alors appel\'ee principe de transport de masse : 
\begin{equation}
\int_{\mathcal S} \int_S f(S,p,q) dq d\mu_\nu(S,p) =  \int_{\mathcal S} \int_S f(S,p,q) dp d\mu_\nu(S,q). 
\label{MTP}
\end{equation}


\subsection{D\'emonstration du th\'eor\`eme \ref{theo_appendix}}

On va d\'emontrer successivement les points (i), (ii) et (iii) du th\'eor\`eme \ref{theo_appendix} : leurs preuves sont similaires mais ind\'ependantes. Pendant toute la d\'emonstration on notera $\mu$ une mesure unimodulaire sur $\mathcal S$, c'est \`a dire le pouss\'e en avant sur $\mathcal S$ d'un IRS via l'application naturelle $\sub_G\to\mathcal S$. 
 
\subsubsection{Bouts isol\'es}

\begin{lem}
Pour $\mu$-presque toute $(S,p)$, si $S$ a un bout non-cuspidal isol\'e des autres bouts non-cuspidaux alors elle a au plus deux bouts non-cuspidaux. 
\end{lem}

\begin{proof}
Supposons que $S$ ait un bout non-cuspidal isol\'e des autres avec une $\mu$-probabilit\'e $>0$ ; il existe alors un $r>0$ tel que, avec $\mu$-probabilit\'e $a>0$, $S-B_S(p,R)$ ait une composante connexe qui ait exactement un bout non-cuspidal : on fixe pour la suite un tel $r$. Soit $f$ la fonction caract\'eristique du sous-ensemble $E$ de ${\mathcal S}_2$ constitu\'e des $(S,p,q)$ telles que $S-B-S(p,r)$ ait au moins trois composantes connexes de volume infini, que $q$ soit contenu dans l'une d'elles, et que cette derni\`ere n'ait de plus qu'un seul bout non-cuspidal. On prouvera plus bas que cet ensemble est bor\'elien, ce qui permet alors d'obtenir une contradiction via le principe de transport de masse comme suit : si l'on fixe $(S,p,q)\in E$ alors l'ensemble des $q'\in S$ tels que $f(S,p,q') = 1$ est de volume infini (puisqu'il contient au moins une composante connexe de volume infini de $S-B_S(p,r)$, celle de $q$). On a donc
$$
\int_{\mathcal S} \int_S f(S,p,q) dq d\mu(S,p) \ge a\times\infty = \infty. 
$$
Mais si l'on consid\`eere l'ensemble des $p'\in S$ tels que $f(S,p',q) = 1$ est de volume inf\'erieur \`a celui d'une boule de rayon $2r$ dans $\HH^2$ : en effet, si $d(p,p') > 2r$ alors soit $B_S(p',r)$ est contenue dans la m\^eme composante connexe de $S-B_S(p,R)$ que $q$, auquel cas $S-B_S(p',r)$ a au plus deux composantes connexes de volume infini, soit il est contenus dans une autre, auquel cas $q$ est reli\'e dans $S-B_S(p',R)$ \`a l'une des composantes de volume infini de $S-B_S(p,r)$, et sa composante dans $S-B_S(p',r)$ a donc au moins deux bouts non-cuspidaux. On a donc 
$$
\int_{\mathcal S} \int_S f(S,p,q) dp d\mu(S,q) \le \vol B_{\HH^2}(2r) < \infty, 
$$
ce qui avec l'\'egalit\'e obtenue pr\'ec\'edemment nie \eqref{MTP}. 

Il reste \`a prouver que l'ensemble $E$ est bien un ensemble bor\'elien. On a $E=E_1\cap E_2$ o\`u $E_1$ est l'ensemble des $(S,p,q)$ telles que $S-B_S(p,r)$ ait au moins trois composantes connexes de volume infini, et $E_2$ l'ensemble des $(S,p,q)$ telles que $q$ soit dans une composante connexe de $S-B_S(p,r)$ ayant exactement un bout non-cuspidal. Montrons que $E_1$ est un bor\'elien : si on fixe un $V>0$, l'ensemble $U_R$ des $(S,p)$ telles que $S-B_S(p,r)$ ait au moins trois composantes connexes de volume $>V$ est ouvert dans le ferm\'e des $(S,p)$ telles que $S-B_S(p,r)$ ait au moins trois composantes connexes, et comme $E_1 = \bigcap_{V\in\NN} U_R$ ce dernier est bien bor\'elien. De la m\^eme mani\`ere, on a $E_2 = \bigcap_{R\in\NN} W_R$ o\`u $W_R$ est l'ensemble ouvert des $(S,p,q)$ tels que $q$ appartienne \`a une composante $C$ de $S-B_S(p,r)$ de volume infini telle que $C-B_S(p,R)$ ait au plus une composante de volume infini, et il suit que $E_2$ aussi est bor\'elien. 
\end{proof}


\subsubsection{Genre}

\begin{lem}
Pour $\mu$-presque toute $(S,p)$, si $S$ n'est pas de genre nul alors elle est de genre infini : en fait tout bout non-cuspidal de $S$ est de genre infini. 
\end{lem}

\begin{proof}
La preuve est similaire \`a celle du lemme pr\'ec\'edent : supposons qu'avec $\mu$-probabilit\'e $>0$ la surface $S$ soit de genre non nul, et ait un bout non-cuspidal de genre nul alors il existe un $r>0$ tel qu'avec une $\mu$-probabilit\'e $>0$ la sous-surface $B_S(p,r)$ soit de genre non nul et l'une des composantes de volume infini de $S-B_S(p,r)$ soit de genre nul. On d\'efinit l'ensemble $E$ comme suit : on a $(S,p,q)\in E$ si et seulement si $B_S(p,r)$ ait genre $>0$ et $q$ appartienne \`a une composante de volume infini et de genre nul de $S-B_S(p,r)$. Si $f$ d\'esigne la fonction caract\'eristique de $E$, de la m\^eme mani\`ere que ci-dessus si l'on prouve que $f$ est bor\'elienne il suit que le c\^ot\'e gauche de \eqref{MTP} appliqu\'ee \`a $f$ est infini tandis que son c\^ot\'e droit est fini (en effet, si $f(S,p,q)=1$ et $d(p,p')>2r$ alors soit $B_S(p',r)$ est contenue dans la m\^eme composante de $S-B_S(p,r)$ que $q$ et est de genre nul, ou est contenue dans une autre composante de $S-B_S(p,r)$ et alors la composante de $q$ dans $S-B_S(p',r)$ contient $B_S(p,r)$ et est donc de genre $>0$). 

Il suffit donc de prouver que $E$ est un ensemble bor\'elien. La condition que $B_S(p,r)$ contienne un lacet non-s\'eparant d\'efinit un ensemble ouvert $U$, de m\^eme pour un $V>0$ celle que $q$ soit dans une composante de volume $>V$ de $S-B_S(p,r)$ d\'efinit un ouvert $W_V$. Enfin, pour $R>r$ la condition que $q$ appartienne \`a une composante de $B_S(p,R)-B_S(p,r)$ de genre nul d\'efinit elle aussi un ouvert $V_R$. On a 
$$
E = U\cap\bigcap_{V\in\NN} W_V \cap\bigcap_{R\in\NN} V_R
$$
et $E$ est donc bien un bor\'elien.  
\end{proof}


\subsubsection{Bouts cuspidaux}

\begin{lem}
Pour $\mu$-presque toute $(S,p)$, si $S$ a un bout cuspidal alors les bouts cuspidaux de $S$ sont denses dans ${\mathcal E}(S)$. 
\end{lem}

\begin{proof}
La preuve suit encore le m\^eme principe que les deux pr\'ec\'edentes, cette fois en utilisant un ensemble $E$ d\'efini comme suit : $(S,p,q)\in E$ si et seulement si :
\begin{itemize}
\item[(i)] la boule ferm\'ee $\ovl{B_S(p,r)}$ intersecte la partie $\eps$-mince d'un cusp de $S$ (o\`u $\eps$ est plus petit que la constante de Margulis de $\HH^2$) ;
\item[(ii)] $S-B_S(p,r)$ a une composante connexe $C$ de volume infini n'ayant aucun bout cuspidal et $q\in C$ 
\end{itemize}
(ici $r$ est choisi de telle sorte que les deux premi\`eres conditions sur $(S,p)$ aient une $\mu$-probabilit\'e positive). 

On v\'erifiera plus bas que $E$ est un bor\'elien, et le c\^ot\'e gauche de \eqref{MTP} appliqu\'e \`a sa fonction caract\'eristique $f$ est clairement infini. Il reste \`a voir que le c\^ot\'e droit est fini : pour cela on fixe $(S,p,q)\in E$, et on montre que le volume des $p'$ avec $(S,p',q)\in E$ est fini, born\'e ind\'ependamment de $(S,p,q)$ : en effet, si $d(p,p')>2r$ on est dans l'un des trois cas suivants :  
\begin{itemize}
\item $p'$ est dans le cusp de $S-B_S(p,r)$ ;
\item $B_S(p',r)$ est contenue dans la m\^eme composante $C$ de $S-B_S(p,r)$ que $q$, et comme celle-ci n'a pas de bout cuspidal, $S-B_S(p',r)$ n'a pas de cusp parmi ses composantes connexes ; 
\item le cusp isol\'e par $B_S(p,r)$ est dans la m\^eme composante connexe de $S-B_S(p',r)$ que $q$. 
\end{itemize} 
Le seul cas o\`u l'on peut \'eventuellement avoir $(S,p',q)\in E$ est le premier, et il suit que 
$$
\vol(p':\: (S,p',q)\in E) \le \vol B_{\HH^2}(2r) + V
$$
o\`u $V$ est le volume de la partie $\eps$-mince d'un cusp. 

Montrons que $E$ est un bor\'elien : la condition (i) est ferm\'ee ; Pour un $R>r$ fix\'e la condition que $B_S(p,R)\cap C$ n'intersecte pas la partie $\eps$-mince d'un cusp de $S$ est ouverte (noter que la composante $C$ de $q$ dans $S-B_S(p,r)$ est localement bien d\'efinie), et la condition (ii) est l'intersection de ces conditions pour $R\in\NN,\, R>r$. 
\end{proof}


\section{Convergence des surfaces de B\'elyi al\'eatoires}
\label{belyi}
D'apr\`es le lemme \ref{geod_BSconv} le th\'eor\`eme \ref{Belyi_conv} est une cons\'equence imm\'ediate du r\'esultat suivant, o\`u l'on notera :
$$
N_R(M) = |\{\text{g\'eod\'esiques de longueur } \le R \text{ sur } M\}|
$$
pour une vari\'et\'e riemannienne $M$. 

\begin{prop}
Soit $S_C$ une surface de B\'elyi al\'eatoire de complexit\'e $n$ et $R>0$. Quand $n\to+\infty$ elle satisfait avec une probabilit\'e tendant vers 1 la propri\'et\'e suivante : 
$$
\frac{N_R(S_C)} {\vol S_C} \xrightarrow[n\to+\infty]{} 0. 
$$
\label{belyi_geod}
\end{prop}

Il faut d'abord d\'emontrer quelques propri\'et\'es g\'eom\'etriques des surfaces al\'eatoires \`a cusps, que l'on d\'eduit des r\'esultats de Bollob\'as sure les graphes d\'ecrits dans \cite{Brooks_Makover}. On ignorera les probl\`emes li\'es \`a l'orientation cyclique autour des sommets (cf. \cite[Section 4]{Brooks_Makover}) qui ne jouent pas de r\^ole dans les r\'esultats que nous utilisons. Rappelons (cf. \cite[Section 3]{Brooks_Makover}) qu'on dit qu'une surface hyperbolique $S=\Gamma\bs\HH^2$ a des cusps plong\'es de largeur $\ell$ si on peut choisir un syst\`eme $\mathcal B$ d'horoboules disjointes autour des points fixes paraboliques de $\Gamma$ qui soit $\Gamma$-invariant et tel que tout parabolique de $\Gamma$ pr\'eservant une horoboule $B\in\mathcal B$ d\'eplace d'au moins $\ell$ sur son bord. 

\begin{lem}
Si $\mathcal G$ est un graphe al\'eatoire suivant la loi de Bollob\'as (d\'ecrite en \ref{belyi_maintext} et dans \cite[Section 5]{Brooks_Makover}) et $S$ la surface al\'eatoire model\'ee sur $\mathcal G$, alors avec probabilit\'e tendant vers 1 quand le nombre de sommets tend vers l'infini on a les propri\'et\'es suivantes :
\begin{itemize}
\item[(i)] Pour un $\ell>0$ fix\'e, $S$ a des cusps plong\'es de largeur $\ell$ ; 
 \item[(ii)] Pour un $R>0$ fix\'e, le nombre de g\'eod\'esiques ferm\'ees de longueur $\le R$ sur $S$ est born\'e.
\item[(iii)] Le nombre de cusps de $S$ est un $o(n)$. 
\end{itemize}
\label{surfaces_graphes}
\end{lem}

\begin{proof}
Le point (i) est un des r\'esultats principaux de \cite{Brooks_Makover} (cf. leur th\'eor\`eme 2.1 et la section 6 de leur article). 

Le point (ii) suit imm\'ediatement des deux faits suivants : 
\begin{itemize}
\item[(a)] Pour un $R\in\NN$ donn\'e le nombre de chemins de longueur $R$ dans $\mathcal G$ est born\'e avec probabilit\'e tendant vers 1 quand $n\to+\infty$ ; 
\item[(b)] On a un isomorphisme $\phi : \pi_1({\mathcal G})\to \pi_1(S)\subset \PSL_2(\RR)$ tel que si $\phi(c)$ est hyperbolique alors $\ell(\phi(c))\ge k\log\ell(c)$ o\`u $k>0$ ne d\'epend pas de $S$ (tant que cette derni\`ere est un rev\^etement de la surface modulaire). 
\end{itemize}
On rappelle que $g\in\PSL_2(\RR)$ est dit hyperbolique si $\ell(g):=\inf_{x\in\HH^2}d(x,gx) >0$ ; sur le c\^ot\'e droit $\ell(c)$ d\'esigne le nombre minimal d'ar\^etes de ${\mathcal G}$ emprunt\'ees par un lacet repr\'esentant $c$. Le point (a) va suivre du fait que les variables al\'eatoires donnant le nombre de {\it circuits} d'une longueur donn\'ee dans $\mathcal G$ sont asymptotiquement des variables de Poisson ind\'ependantes \cite[2.4]{Bollobas_book} (voir aussi \cite[Th\'eor\`eme 5.3]{Brooks_Makover}). Pour en d\'eduire le r\'esultat sur les chemins il suffit d'observer que si $\mathcal G$ est un graphe trivalent quelconque, pour un $R>0$ donn\'e un circuit de longueur $\le R$ n'est contenu que dans un nombre fini de chemins de longueur $\le R$, et ce nombre ne d\'epend que de $R$. Le point (b) est une cons\'equence de ce que le plongement $\PSL_2(\ZZ)\to\HH^2$ donn\'e par une application orbitale $\gamma\mapsto\gamma x_0$ ait une distortion logarithmique (ceci suit facilement du fait que si l'on tronque les cusps on obtient une quasi-isom\'etrie de $\PSL_2(\ZZ)$ sur $\HH^2$ priv\'e d'une union d'horoboules disjointes, et que la distortion de l'inclusion de cette derni\`ere dans $\HH^2$ est logarithmique). 

Enfin, le point (iii) dans l'  \'enonc\'e est une cons\'equence du point (ii), du lemme \ref{Farber} et du fait que dans toute suite de surfaces qui soit BS-convergente vers $\HH^2$ le nombre de cusps est n\'egligeable par rapport au volume (vu que la contribution de la partie $\eps$-mince d'un cusp ($\eps = $ constante de Margulis) au volume de la partie mince est la m\^eme pour tous). Notons que la distribution plus pr\'ecise du nombre de cusps a fait l'objet de plusieurs travaux suivant \cite{Brooks_Makover} (citons par exemple \cite{Gamburd_Makover}). 
\end{proof}

Le r\'esultat suivant est d\^u \`a Brooks \cite[Lemma 3.1]{Brooks_platonic}. 

\begin{lem}[Brooks]
Il existe un $\ell$ tel que si $S$ est une surface hyperbolique ayant des cusps plong\'es de largeur $\ell$ alors pour tout $R>0$ on a 
$$
N_R(S_C) \le N_{2R}(S). 
$$
\label{comp_geod}
\end{lem}

Il ne reste plus qu'\`a d\'eduire la proposition \ref{belyi_geod} : tout d'abord on note que par le point (iii) du lemme \ref{surfaces_graphes} on a $\vol(S_C)\sim\vol(S)$ avec une probabilit\'e tendant vers 1 : en effet, vu que $S$ et $S_C$ ont (par d\'efinition de la derni\`ere) le m\^eme genre, notant $h$ le nombre de cusps de $S$ on a $\vol(S) = \vol(S_C) + 2\pi h$ par le th\'eor\`eme de Gauss-Bonnet. Par le point (i) du m\^eme lemme on peut appliquer le lemme \ref{comp_geod} \`a $S$ avec probabilit\'e tendant vers 1. On a donc (pour $n$ assez grand) avec probabilit\'e tendant vers 1 la majoration
$$
\frac{N_R(S_C)} {\vol S_C} \le 2\frac{N_{2R}(S)} {\vol S} 
$$
pour une surface de B\'elyi de complexit\'e $n$. La proposition \ref{belyi_geod} suit alors du point (ii) du lemme \ref{surfaces_graphes}. 


\bibliographystyle{plain}
\bibliography{bib}

\end{document}